\documentclass[11pt, twoside, leqno]{article}

\usepackage{amssymb}
\usepackage{amsmath}
\usepackage{amsthm}
\usepackage{color}
\usepackage{mathrsfs}
\usepackage{txfonts}
\usepackage{indentfirst}

\allowdisplaybreaks

\def\red{\color{red}}

\def\rr{{\mathbb R}}
\def\rn{{\mathbb{R}^n}}
\def\zz{{\mathbb Z}}
\def\cc{{\mathbb C}}

\def\nn{{\mathbb N}}

\def\cf{{\mathcal F}}

\def\fz{\infty }
\def\az{\alpha}
\def\bz{\beta}

\def\ez{\epsilon}
\def\gz{{\gamma}}

\def\lz{\lambda}

\def\lf{\left}
\def\r{\right}

\def\ls{\lesssim}

\def\noz{\nonumber}
\def\wz{\widetilde}

\def\loc{{\mathop\mathrm{\,loc\,}}}
\def\supp{\mathop\mathrm{\,supp\,}}

\def\XXint#1#2#3{{\setbox0=\hbox{$#1{#2#3}{\int}$ }
\vcenter{\hbox{$#2#3$ }}\kern-.6\wd0}}

\newtheorem{theorem}{Theorem}[section]
\newtheorem{lemma}[theorem]{Lemma}
\newtheorem{corollary}[theorem]{Corollary}
\newtheorem{proposition}[theorem]{Proposition}

\theoremstyle{definition}
\newtheorem{remark}[theorem]{Remark}

\renewcommand{\appendix}{\par
   \setcounter{section}{0}%
   \setcounter{subsection}{0}%
   \setcounter{subsubsection}{0}%
   \gdef\thesection{\@Alph\c@section}%
   \gdef\thesubsection{\@Alph\c@section.\@arabic\c@subsection}%
   \gdef\theHsection{\@Alph\c@section.}%
   \gdef\theHsubsection{\@Alph\c@section.\@arabic\c@subsection}%
   \csname appendixmore\endcsname
 }

\numberwithin{equation}{section}

\begin{document}

\arraycolsep=1pt

\title{\bf  Functional Calculus  on BMO-type Spaces of Bourgain,  Brezis and  Mironescu
\footnotetext{\hspace{-0.35cm} 2010 {\it
Mathematics Subject Classification}. Primary 46E35;
Secondary 42B35, 46E30.
\endgraf {\it Key words and phrases.}  superposition operator, ${\rm B}$ space, $\mathop\mathrm{BMO}\,$ space, ${\rm VMO}$ space.
\endgraf This project is supported by the National
Natural Science Foundation of China
(Grant Nos.  11771446, 11761131002, 11571039, 11726621 and  11471042).}}
\author{Liguang Liu,  Dachun Yang and Wen Yuan\footnote{Corresponding author/{\red August 13, 2018}/Final version.}}
\date{}
\maketitle

\vspace{-0.8cm}

\begin{center}
\begin{minipage}{13cm}
{\small {\bf Abstract}\quad
A nonlinear superposition operator $T_g$ related to
a Borel measurable function $g:\ {\mathbb C}\to {\mathbb C}$ is defined via $T_g(f):=g\circ f$ for
any complex-valued function $f$ on ${\mathbb R^n}$.
This article is devoted to investigating the mapping properties of $T_g$ on a new
BMO type space recently introduced by
Bourgain,  Brezis and  Mironescu [J. Eur. Math. Soc. (JEMS) 17 (2015),  2083-2101],
as well as its VMO and CMO type subspaces.
Some sufficient and necessary conditions for the inclusion  and the
continuity properties of $T_g$ on these spaces are obtained.
}
\end{minipage}
\end{center}

%\tableofcontents

\section{Introduction}\label{s1}

Recently, Bourgain,  Brezis and  Mironescu \cite{bbm} introduced a new BMO type space $\mathrm{B}$ on the unit cube, which is large enough to
contain  the  BMO space, the space BV of functions of bounded variation and the Sobolev  space $W^{1/p,p}$ with
$p\in(1,\fz)$ as its subspaces, and meanwhile it is also small enough to ensure that any integer-valued element belonging to its VMO type subspace $\mathrm{B}_0$ is necessarily constant. Notice that the implication
 property $$``f\in X\ \mbox{being integer-valued}\Longrightarrow f=\mbox{constant almost everywhere}"$$
is already known to be true if $X$ is the VMO space or the Sobolev space $W^{1/p,p}$ with
$p\in[1,\fz)$, which are both subspaces of $\mathrm{B}_0$.
In \cite{abbf}, Ambrosio,  Bourgain,  Brezis and Figalli further found an interesting connection between the $\mathrm{B}$ space
and the notion of perimeter of sets. Indeed, via a global version of the norm of the
new BMO type space $\mathrm{B}$, they found a new characterization of perimeter of sets independent of the theory of distributions.

In view of these remarkable applications of the space $\mathrm{B}$ in analysis and geometry, it would be interesting to
explore more properties or characterizations of it. The main aim of this article is to clarify the mapping properties
of the nonlinear superposition operator on the space $\mathrm{B}$
as well as the corresponding VMO and CMO type subspaces. Recall that a \emph{superposition operator $T_g$} (also called Nemytskij
operator) related to
a Borel measurable function $g:\ \cc\to \cc$ is  given by
\begin{equation}\label{superp}
T_g(f):=g\circ f\qquad \mbox{for any complex-valued function}\ f.
\end{equation}
This nonlinear operator $T_g$ appears frequently in various branches of mathematics, which
plays a crucial role in nonlinear analysis and can be applied to ordinary or partial
differential equations, physics and engineering;
see, for example, \cite{be09,bu10,bu12,pkr} for some of its recent applications.

The study
of the behavior of the superposition operator $T_g$ on
function spaces has a long history.
Early works are on the behavior of $T_g$ on Sobolev spaces, which was due to Marcus and Mizel \cite{mm72,mm79,mm79b}.
In  \cite{az}, Appell and Zabrejko studied
the action of $T_g$ on Lebesgue, Orlicz and H\"older spaces.
During the last three decades,
important progresses have been made on the study of
superposition operators on function spaces with fractional-order of smoothness (such as Sobolev spaces, H\"older-Zygmund spaces, Besov  and Triebel-Lizorkin spaces), mainly due to Bourdaud
and Sickel et al.
For instance, we refer the reader to \cite{b91,bm91,b92a,bk92,rs96,s96,s97} for Sobolev spaces, to  \cite{s89,b92,b93,bk95,rs96,s98a,s98b,bms08,bs10,bs11,bms14}
for Besov and Triebel-Lizorkin spaces, to \cite{bl02} for H\"older-Zygmund spaces, to \cite{bls05,bls06}
for spaces of functions of
bounded $p$-variation  and
to \cite{f85,bn,c,b02, bls} for classical BMO-type spaces. More historical information was given in \cite{az10}.
Of particular importance to us is the article \cite{bls} of Bourdaud,  Lanza de Cristoforis
and  Sickel, which provides a nearly complete picture on the mapping properties of
superposition operators on BMO and its subspaces VMO  and CMO on $\rn$.
Based on these, it is natural to study the behavior of $T_g$
on the aforementioned new BMO type space $\mathrm{B}$
introduced in \cite{bbm,abbf} as well as its VMO and CMO type subspaces.

To state the main results of this article, we begin with some basic notation and notions.
For any $r\in(0,\infty)$ and $a\in\rn$, let
$Q_r(a):=Q(a,r)$ be the open cube centered at $a$ with its sides parallel
 to the coordinate axes and of side length $r$.
 Such an open cube with side length $r$ is called an \emph{$r$-cube}.
Given a cube $Q\subset\rn$ and a complex-valued  locally integrable function $f$ defined on $\rn$,
we let
$$M(f,Q):=\fint_Q|f(x)-f_Q|\,dx,$$
where
\begin{align*}
\fint_Q:=\frac1{|Q|}\int_Q\qquad \textup{and}\qquad f_Q:=\fint_Q f(x)\,dx.
\end{align*}
Let $Q_0:=(0,1)^n$ be the unit open cube of $\rn$.
Denote by $L^1(Q_0)$ the set of all complex-valued measurable functions $f$ on $\rn$ such that
$\int_{Q_0}|f(x)|\,dx<\infty$.
For any $f\in L^1(Q_0)$ and $\ez\in(0,1)$, let
$$[f]_{\ez,Q_0}:=\sup_{\cf_\ez}\lf\{\ez^{n-1}\sum_{j\in J} M(f,Q_\ez(a_j))\r\},$$
where the supremum is taken over all collections $\cf_\ez:=\{Q_\ez(a_j)\}_{j\in J}$ of mutually disjoint
$\ez$-cubes in $Q_0$ with sides parallel to the coordinate axes of $\rn$ satisfying  $\#\cf_\ez=\#J\le 1/\ez^{n-1}$.
Here and hereafter, for any set $E$, we use $\#E$ to denote
its \emph{cardinality}. The BMO type space $\mathrm{B}(Q_0)$ is  defined as the collection of all $f\in L^1(Q_0)$
such that
$$\|f\|_{\mathrm{B}(Q_0)}:=\int_{Q_0}|f(x)|\,dx+\sup_{0<\ez<1}[f]_{\ez,Q_0}<\infty.$$
We point out that this BMO type space $\mathrm{B}(Q_0)$,  denoted originally by $\mathrm{B}$ in \cite{bbm},
was equipped with the  seminorm $\|f\|_{\mathrm{B}}:= \sup_{0<\ez<1}[f]_{\ez,Q_0}$ therein, which makes $\mathrm{B}$ into a Banach space modulo the space of constant functions.
Since the operator $T_g$ is not defined on the quotient space, we use the norm $\|\cdot\|_{\mathrm{B}(Q_0)}$ instead of $\|\cdot\|_{\mathrm{B}}$
throughout this article.

Recall that the classical space  $\mathop\mathrm{BMO}\,(Q_0)$ is defined to be the set of all complex-valued
locally integrable functions on $Q_0$ such that
$$\|f\|_{\mathop\mathrm{BMO}\,(Q_0)}:=\sup_{Q\subset Q_0}M(f,Q)<\fz,$$
where the supremum is taken over all cubes $Q$ in $Q_0$.
Based on \cite[p.\,2084]{bbm}, one has $\mathop\mathrm{BMO}\,(Q_0)=\mathrm{B}(Q_0)$ when $n=1$
 and $\mathop\mathrm{BMO}\,(Q_0)\subsetneqq\mathrm{B}(Q_0)$ when $n>1$.

For any two vector spaces $X$ and $Y$, the symbol $X\subset Y$ only means that $X$ is a subset of $Y$, and $X\hookrightarrow Y$ means
that not only $X\subset Y$ but also the embedding from $X$ into $Y$ is continuous.
If we let $\mathrm{B}_c(Q_0)$ be the closure of  $C_c^\fz(Q_0)$  in $\mathrm{B}(Q_0)$,
and $\mathrm{B}_0(Q_0)$ the set of all  $f\in \mathrm{B}(Q_0)$ such that
$$\limsup_{\ez\to 0}[f]_{\ez,Q_0}=0,$$
then it is easy to show that
$\mathrm{B}_c(Q_0)\hookrightarrow \mathrm{B}_0(Q_0)$.
It is also easy to see that $$\mathop\mathrm{VMO}\,(Q_0):=\lf\{f\in \mathop\mathrm{BMO}\,(Q_0):\ \lim_{\ez\to 0}\sup_{{Q\subset Q_0, \ell(Q)\le \ez}}M(f,Q)=0\r\}\subset \mathrm{B}_0(Q_0)$$
and
$$\mathop\mathrm{CMO}\,(Q_0):=
\lf\{\textup{the closure of}\; C_c^\fz(Q_0) \;\textup{in}\;
\mathop\mathrm{BMO}\,(Q_0)\r\}\subset
 \mathrm{B}_c(Q_0),$$
 where  $\ell(Q)$ always denotes the side length of a cube $Q$.

An analogous  global version of $B(Q_0)$ can be defined as follows.
Given a complex-valued locally integrable function $f$ on $\rn$ and $\ez\in(0,1)$, define
$$[f]_{\ez}:=\sup_{\cf_\ez}\lf\{\ez^{n-1}\sum_{j\in J} M(f,Q_\ez(a_j))\r\},$$
where the supremum is now taken over all collections $\cf_\ez:=\{Q_\ez(a_j)\}_{j\in J}$ of mutually disjoint
$\ez$-cubes  in $\rn$ with sides parallel to the coordinate axes and $\#\cf_\ez=\#J\le 1/\ez^{n-1}$.
Denote by $L^1_\loc(\rn)$ the set of all locally integrable functions on $\rn$ and $\mathrm{B}(\rn)$
the space of all complex-valued functions $f\in L^1_\loc(\rn)$
such that
$$\|f\|_{\mathrm{B}(\rn)}:=\sup_{|Q|=1}\int_{Q}|f(x)|\,dx+\sup_{0<\ez<1}[f]_\ez<\infty,$$
where the first supremum is taken over all $1$-cubes $Q$ in $\rn$ with sides parallel to the coordinate axes.
By this definition, it is easy to see that   $\mathrm{B}(\rn)$ is translation invariant.

Here, it should be mentioned that  the limit when $\ez\to 0$ of an isotropic
variant $I_\ez(f)$ of  $[f]_\ez$, defined via removing the restriction ``sides parallel
to the coordinate axes" from the definition of $[f]_\ez$, was used in \cite{abbf}
to give a new characterization of the perimeter of sets,
independent of the theory of distributions.  More precisely, it was proved in \cite[Theorem 1]{abbf} that, for any measurable set $A\subset \rn$,
it holds true that $\lim_{\ez\to0}I_\ez(\mathbf{1}_A)=\frac12\min\{1,P(A)\}$, where $\mathbf{1}_A$ denotes the characteristic function on $A$
and $P(A)$ the perimeter of $A$.

Let us clarify some obvious relations among $\mathrm{B}(\rn)$ and the classical BMO type spaces on $\rn$.
Recall that
$$\mathop\mathrm{bmo}\,(\rn)
:=\lf\{
f\in \mathop\mathrm{BMO}\,(\rn):\, \|f\|_{\mathop\mathrm{bmo}\,(\rn)}:=\|f\|_{\mathop\mathrm{BMO}\,(\rn)}+\sup_{|Q|=1}\int_Q|f(x)|\,dx<\fz
\r\}.$$
Let
$
\mathop\mathrm{cmo}\,(\rn)
$
be
the closure of $C_c^\fz(\rn)$ of smooth functions with compact
supports in
$\mathop\mathrm{bmo}\,(\rn)$
and
$$\mathop\mathrm{vmo}\,(\rn):=\lf\{f\in \mathop\mathrm{bmo}\,(\rn):\ \lim_{\ez\to 0}\sup_{{\ell(Q)\le \ez}}M(f,Q)=0\r\}.$$
Analogously, define
$$
\mathrm{B}_c(\rn)
:= \lf\{
\textup{the closure of}\; C_c^\fz(\rn)\;\textup{in}\;
 \mathrm{B}(\rn)\r\}
\quad
\textup{and}
\quad
\mathrm{B}_0(\rn)
:= \lf\{f\in \mathrm{B}(\rn):\, \limsup_{\ez\to 0}[f]_\ez=0\r\}.
$$
From these definitions, we deduce that
$\mathop\mathrm{bmo}\,(\rn)\subset \mathrm{B}(\rn)$, $\mathop\mathrm{vmo}\,(\rn)\subset \mathrm{B}_0(\rn)$ and
$\mathop\mathrm{cmo}\,(\rn)\subset \mathrm{B}_c(\rn).$
It is also easy to see that $\mathop\mathrm{bmo}\,(\rr)=\mathrm{B}(\rr)$, $\mathop\mathrm{vmo}\,(\rr)= \mathrm{B}_0(\rr)$ and
$\mathop\mathrm{cmo}\,(\rr)= \mathrm{B}_c(\rr).$

The first result of this article reads as follows.

\begin{theorem}\label{t-bmo}
The following five statements are equivalent:
\begin{enumerate}
\item[{\rm(i)}] $\sup_{x,y\in \cc}(1+|x-y|)^{-1}|g(x)-g(y)|<\fz$;
\item[{\rm(ii)}]  $T_g(\mathrm{B}(\rn)) \subset \mathrm{B}(\rn)$;
\item[{\rm(iii)}] $T_g(\mathrm{B}_c(\rn)) \subset \mathrm{B}(\rn)$;
\item[{\rm(iv)}]  $T_g(\mathrm{B}(Q_0)) \subset \mathrm{B}(Q_0)$;
\item[{\rm(v)}] $T_g(\mathrm{B}_c(Q_0)) \subset \mathrm{B}(Q_0)$.
\end{enumerate}
Moreover, if any of the above holds true, then $T_g$ maps bounded subsets of $\mathrm{B}(\rn)$ (resp., $\mathrm{B}(Q_0)$)
to bounded subsets of $\mathrm{B}(\rn)$  (resp., $\mathrm{B}(Q_0)$).
\end{theorem}

Comparing with \cite[Theorem 1]{bls}, we find that the condition on $g$
which ensures the inclusion $T_{g}(\mathrm{B}(\rn))\subset \mathrm{B}(\rn)$ here is the same as that for
$T_{g}(\mathop\mathrm{BMO}\,(\rn))\subset \mathop\mathrm{BMO}\,(\rn)$ and
$T_{g}(\mathop\mathrm{bmo}\,(\rn))\subset \mathop\mathrm{bmo}\,(\rn)$.

The proof of the implication  (ii)$\Longrightarrow$(i) in Theorem \ref{t-bmo}
can be reduced to the case $n=1$, and
then one can  employ  the fact  $\mathrm{B}(\rr)=\mathop\mathrm{bmo}\,(\rr)$ and the known
result for $\mathop\mathrm{bmo}\,(\rr)$ obtained in \cite[Theorem 1]{bls}.
Being precisely, for any $f_0\in \mathop\mathrm{bmo}\,(\rr)=\mathrm{B}(\rr)$, define
\begin{equation}\label{extend}
f(x_1,\ldots,x_n):=f_0(x_1),\qquad \forall\; x:=(x_1,\ldots,x_n)\in\rn.
\end{equation}
Then $f_Q=(f_0)_{Q^{1}}$ for any cube $Q$ in $\rn$, where $Q^1$ is the projection interval of $Q$ onto the $x_1$-axis.
Notice that  the assumption (ii)
gives  $T_g f=g\circ f \in \mathrm{B}(\rn)$.
Since $g\circ f_0(x_1)=g\circ f(x_1,\ldots,x_n)$ for any $x:=(x_1,\ldots,x_n)\in\rn$, we then deduce that $T_g f_0\in \mathrm{B}(\rr)=\mathop\mathrm{bmo}\,(\rr)$ and hence $T_g(\mathop\mathrm{bmo}\,(\rr))\subset \mathop\mathrm{bmo}\,(\rr)$.
By employing \cite[Theorem 1]{bls}, we find that $g$ satisfies (i).
One key point for this argument is the inclusion $\mathrm{B}(\rr)\subset \mathrm{B}(\rn)$ in the sense of \eqref{extend}.

Such a simple argument also works for the proof of (iv) $\Longrightarrow$ (i).
However, it can not be applied to prove (iii) $\Longrightarrow$ (i) or (v) $\Longrightarrow$ (i).
This is mainly because  the counterpart of the inclusion $\mathrm{B}(\rr)\subset \mathrm{B}(\rn)$ (in the sense of \eqref{extend})
fails for the ${\rm CMO}$-like spaces $\mathrm{B}_c(\rn)$ and $\mathrm{B}_c(Q_0)$ when $n\ge 2$.
Indeed, the proof of (iii) $\Longrightarrow$ (i) and  (v) $\Longrightarrow$ (i) need a much more detailed and complicated argument,
which  not only
provides \emph{an alterative way} to show (ii) $\Longrightarrow$ (i) and (iv) $\Longrightarrow$ (i), but also
provides some byproducts (see, for example, Proposition \ref{p-pm} and Lemma \ref{l-dilation} below) that are helpful for us to
understand the space $\mathrm{B}$.

In analogy to Theorem \ref{t-bmo}, we give the corresponding results for the ${\rm VMO}$-like spaces $\mathrm{B}_0$.

\begin{theorem}\label{t-vmo}
The following five statements are equivalent:
\begin{enumerate}
\item[{\rm(i)}]  $g$ is uniformly continuous;
\item[{\rm(ii)}]  $T_g(\mathrm{B}_0(\rn)) \subset \mathrm{B}_0(\rn)$;
\item[{\rm(iii)}] $T_g(\mathrm{B}_c(\rn)) \subset \mathrm{B}_0(\rn)$;
\item[{\rm(iv)}]  $T_g(\mathrm{B}_0(Q_0)) \subset \mathrm{B}_0(Q_0)$;
\item[{\rm(v)}] $T_g(\mathrm{B}_c(Q_0)) \subset \mathrm{B}_0(Q_0)$.
\end{enumerate}
Moreover, if any of the above holds true, then $T_g$ maps bounded subsets of $\mathrm{B}_0(\rn)$  (resp., $\mathrm{B}_0(Q_0)$)
to bounded subsets of $\mathrm{B}_0(\rn)$  (resp., $\mathrm{B}_0(Q_0)$).
\end{theorem}

Due to $\mathop\mathrm{vmo}\,(\rr)=\mathrm{B}_0(\rr)\subset \mathrm{B}_0(\rn)$ in the sense of \eqref{extend},
the implication  (ii) $\Longrightarrow$ (i) of Theorem \ref{t-vmo} can be proved
in the same way as that of Theorem \ref{t-bmo}. However, the results concerning the
${\rm CMO}$-like spaces $\mathrm{B}_c(\rn)$  and $\mathrm{B}_c(Q_0)$ also cannot be resolved so trivially.  The details can
be found in Section 4.

When the target spaces become $\mathrm{B}_c(\rn)$ or $\mathrm{B}_c(Q_0)$, we have the following result.

\begin{theorem}\label{t-cmo}
\begin{enumerate}
\item[{\rm(a)}] $T_g(\mathrm{B}_c(\rn))\subset \mathrm{B}_c(\rn)$ if and only if
$g$ is uniformly continuous and $g(0)=0$.

\item[{\rm(b)}]  $T_g(\mathrm{B}_c(Q_0))\subset \mathrm{B}_c(Q_0)$ if and only if
$g$ is uniformly continuous.
\end{enumerate}
\end{theorem}

We point out that the condition for $T_g(\mathrm{B}_c(\rn))\subset \mathrm{B}_c(\rn)$ in
Theorem \ref{t-cmo} is same as that for $T_g(\mathop\mathrm{cmo}\,(\rn))\subset \mathop\mathrm{cmo}\,(\rn)$ in
\cite[Corollary 1]{bls}.
One key tool to prove  Theorem \ref{t-cmo} is  the continuity of $T_g$
at $f\in \mathrm{B}_0(\rn)$  (resp., $\mathrm{B}_0(Q_0)$) as a map from $\mathrm{B}(\rn)$ (resp., $\mathrm{B}(Q_0)$)
to itself, whenever $g$ is uniformly continuous  (see Proposition \ref{p-bmo-c} below).
This continuity result, together with Theorems \ref{t-vmo} and \ref{t-cmo}, also easily implies the following theorem
on the continuity of $T_g$.

\begin{theorem}\label{t-con}
\begin{enumerate}
\item[{\rm(a)}]
 The following are equivalent:
\begin{enumerate}
\item[{\rm(i)}]  $g$ is uniformly continuous;
\item[{\rm(ii)}]  $T_g$ is continuous from $\mathrm{B}_0(\rn)$ to $\mathrm{B}_0(\rn)$;
\item[{\rm(iii)}] $T_g$ is continuous from  $\mathrm{B}_0(Q_0)$ to $\mathrm{B}_0(Q_0)$;
\item[{\rm(iv)}]  $T_g$ is continuous from $\mathrm{B}_c(Q_0)$ to $\mathrm{B}_c(Q_0)$.
\end{enumerate}

\item[{\rm(b)}] $T_g$ is continuous from $\mathrm{B}_c(\rn)$ to $\mathrm{B}_c(\rn)$ if and only if
$g$ is uniformly continuous and $g(0)=0$.
\end{enumerate}
\end{theorem}

When the target space is $\mathrm{B}(\rn)$, the uniformly continuity of $g$ is no longer enough
to ensure the continuity of $T_g$. Instead,
we have the following conclusion, which can also be proved via reducing to the one-dimensional case,
as in the proof of (ii) $\Longrightarrow$ (i) in Theorem \ref{t-bmo}.

\begin{theorem}\label{t-con2}
The operator $T_g$ is continuous from $\mathrm{B}(\rn)$ to $\mathrm{B}(\rn)$
if and only if
$g$ is $\rr$-affine, that is, $g(z)$ is of form $\az z+\bz$ for some
complex numbers $\az$ and $\bz$ and for any $z\in\cc$.
\end{theorem}

The organization of this article is as follows. As preparatory work for proving main theorems,
in Section \ref{s1++}, we  establish a grouping lemma  (see Lemma \ref{part})
which provides a suitable way to enlarge and grouping cubes in order to fit  the definition of  $\mathrm{B}$ spaces.
A consequent application of Lemma \ref{part} is given in  Proposition \ref{prop-new}, in which
we obtain some uniformly estimates of integral averages for functions in $\rm B(\rn)$ and $\rm B(Q_0)$.
Using these results in Section \ref{s1++}, we give the proof of Theorem \ref{t-bmo} in Section \ref{s1+},
in which we also need several auxiliary lemmas,
including a result about the pointwise multipliers on the BMO-type spaces.
The proofs of Theorems \ref{t-vmo},
\ref{t-cmo} and    \ref{t-con2} are presented, respectively,  in Sections 4, 5 and 6.
Since the structure of $\mathrm{B}$ space is much more complicated than that of BMO,
comparing with the arguments in \cite{bls} for the classical BMO spaces,
we need  more subtle and sophisticated arguments in this article
(see, for example, Lemma \ref{part} and Proposition \ref{prop-new}).

Throughout this article, let $\nn:=\{1,2,\ldots,\}$ and $\zz:=\{0,\pm1,\dots\}$.
We use $C$ to denote
a {\it positive constant} that is
independent of the main parameters involved but
whose value may differ from line to line.
Sometimes we use $C_{(\alpha,\beta,\ldots)}$ to indicate that a constant $C$ depends on the
given parameters $\az$, $\bz$, $\ldots$.
If $f\le Cg$, we then write $f\ls
g$ and, if $f\ls g\ls f$, we write $f\sim g$.
For any $s\in\rr$, denote by $\lfloor s\rfloor$ the largest integer not greater than $s$.
For any cube $Q$ in $\rn$, the notation $\ell(Q)$ denotes the side length of $Q$.
For any $\lz\in(0,\infty)$ and any cube $Q$ in $\rn$, denote by
$\lz Q$ the cube with the same center as that of $Q$ but of side length $\lz\ell(Q)$.
Also, for any set $E$, we use $\#E$ to denote its cardinality.

\section{A grouping lemma}\label{s1++}

Let us begin with the following grouping lemma.
For any $j\in\zz$ and $k\in\zz^n$, let $Q_{j,k}$ denote the dyadic cube $2^{-j}([0,1)^n+k)$.
Denote by $\mathcal{Q}$ the collection of all dyadic cubes and $\mathcal{Q}_j:=\{Q_{j,k}\}_{k\in\zz^n}$.

\begin{lemma}\label{part}
Let $k_0\in\nn$ and $k_0\ge 2$.
\begin{enumerate}
\item[{\rm (a)}]
Let $\{Q_i\}_{i\in J}$ be a family of mutually disjoint open $2^{-k_0}$-cubes in $\rn$ with $\# J\le 2^{k_0(n-1)}$.
For each $i\in J$, let $\wz Q_i:=2Q_i$, which is of  side length $2^{-k_0+1}$ and $\wz Q_i\supset {Q}_i$.
Then there exists  a positive integer $N:=N(n)\le 2^n$ such that the  cubes $\{\wz Q_i\}_{i\in J}$ enjoy the following properties:
\begin{enumerate}
\item[\rm (i)]  $J=J^1\bigcup \cdots \bigcup J^N$;
\item[\rm (ii)] for every $j\in\{1,\ldots,N\}$, the cubes $\{\wz{Q}_i\}_{i\in J^j}$ are mutually disjoint;
\item[\rm (iii)] for every $j\in\{1,\ldots,N\}$, the cardinality $\#J^j\le 2^{(k_0-1)(n-1)}$.
\end{enumerate}

\item[{\rm (b)}]
Let $\{Q_i\}_{i\in J}$ be a family of mutually disjoint dyadic cubes in $Q_0$ with side length $2^{-k_0}$ and $\# J\le 2^{k_0(n-1)}$.
For each $i\in J$, let $\wz Q_i$ be the unique  dyadic cube with side length $2^{-k_0+1}$ contained in $\overline Q_0$ such that $Q_i \subset \wz Q_i$.
Then  there exists  a positive integer $N=N(n)\le  2^n$ such that the items (i), (ii) and (iii) in (a) remain true.
\end{enumerate}
\end{lemma}

\begin{proof}
First we show (a).  Since all $\wz Q_i$ are open, we know that any point in $\rn$
can be covered by at most $2^n$ elements from $\{\wz Q_i\}_{i\in J}$,  due to the non-overlapping property of $\{Q_i\}_{i\in J}$.
With this observation, the grouping procedure can be done as follows.
Put the index $i=1$ in $J_1$. If $\wz Q_2$ does not intersect $\wz Q_1$ and $\sharp J_1< 2^{(k_0-1)(n-1)}$, then we put the index $i=2$ in $J_1$; otherwise we put the index $i=2$ in $J_2$. Next, we look at $\wz Q_3$ and consider three cases:
\begin{itemize}
\item If $\wz Q_3$ does not intersect $\wz Q_1$ and $\sharp J_1<  2^{(k_0-1)(n-1)}$, then put the index $i=3$ in $J_1$.
\item If $\wz Q_3$ intersects $\wz Q_1$ or $\sharp J_1=  2^{(k_0-1)(n-1)}$,  but $\wz Q_3$ does not intersect $\wz Q_2$
and $\sharp J_2< 2^{(k_0-1)(n-1)}$, then put the index $i=3$ in $J_2$.
\item If $\wz Q_3$ intersects $\wz Q_1$ or $\sharp J_1=  2^{(k_0-1)(n-1)}$, meanwhile $\wz Q_3$ intersects $\wz Q_2$ or $\sharp J_2= 2^{(k_0-1)(n-1)}$, then put the index $i=3$ in $J_3$.
\end{itemize}
Continuing the above procedure, we  can
divide $\{\wz Q_i\}_{i\in J}$ into at most $N\ (\le 2^n)$ groups,
$\{\wz Q_i\}_{i\in J_1}$, $\ldots$, $\{\wz Q_i\}_{i\in J_N}$, so that each group is a collection of mutually disjoint cubes
with cardinality not more than $ 2^{(k_0-1)(n-1)}$.

Now, we show (b).  By the geometric properties of dyadic cubes, we know that, if $Q_i$ is a dyadic cube contained in $Q_0$ with side length $\le 1/2$,
then the unique dyadic cube $\wz Q_i$ containing $Q$ with side length $2\ell(Q_i)$ is contained in $\overline{Q}_0.$
In this case, when $i\neq j$, it might happen that  $\wz Q_i=\wz Q_j$.
Also, a dyadic cube $\wz Q_i$ can serves as the $2$-times dyadic extension of at most $2^n$ dyadic cubes in $\{Q_i\}_{i\in J}$.
Based on these observations, following the same grouping procedure as in (a), we immediately obtain
the desired conclusion of Lemma \ref{part}(b).
This finishes the proof of Lemma \ref{part}.
\end{proof}

Observe that the supremum over $\ez\in(0,1)$ in $\|\cdot\|_{\mathrm{B}(Q_0)}$ and $\|\cdot\|_{\mathrm{B}(\rn)}$
can be equivalently taken over $\{2^{-k}:\ k\in\nn\}$.

\begin{lemma}\label{l-eqv}
There exists a positive constant $C:=C_{(n)}$ such that
$$C^{-1}\|f\|_{\mathrm{B}(\rn)}
\le \sup_{|Q|=1}\fint_{Q}|f(x)|\,dx+\sup_{k\in\nn}[f]_{2^{-k}}\le \|f\|_{\mathrm{B}(\rn)},\quad \forall\; f\in \mathrm{B}(\rn)$$
and
$$C^{-1}\|f\|_{\mathrm{B}(Q_0)}
\le \fint_{Q_0}|f(x)|\,dx+\sup_{k\in\nn}[f]_{2^{-k},Q_0}\le \|f\|_{\mathrm{B}(Q_0)},\quad \forall\; f\in \mathrm{B}(Q_0).$$
\end{lemma}

\begin{proof}
By similarity, we only show the second equivalence relation under the assumption of $f\in B(Q_0)$.
To this end, it suffices to consider the first inequality, since the second one is trivial.

Let $f\in B(Q_0)$. If $\ez\in (0,1/2]$, then there exists $k\in\nn$ such that $2^{-k-1}<\ez\le 2^{-k}$. For any $\ez$-cube $Q_\ez$ in $Q_0$,
there exists a $2^{-k}$-cube $Q\subset Q_0$ containing $Q_\ez$, which implies that
$$M(f,Q_\ez)\le 2^n M(f,f_Q)+|f_Q-f_{Q_\ez}|\le 2^{n+1} M(f,f_Q).$$
If $\ez\in (1/2,1)$, then
\begin{align*}
M(f,Q_\ez)\le 2 \fint_{Q_\ez} |f(x)|\,dx \le 2^{n+1} \fint_{Q_0}|f(x)|\,dx.
\end{align*}
We therefore obtain $[f]_{\ez,Q_0}\ls \sup_{k\in\nn}[f]_{2^{-k},Q_0}+\int_{Q_0}|f(x)|\,dx$, as desired.
This finishes the proof of Lemma \ref{l-eqv}.
\end{proof}

Applying Lemmas \ref{part} and \ref{l-eqv}, we have  the following estimates of  functions in $B(\rn)$ and $B(Q_0)$.

\begin{proposition}\label{prop-new}
There exists a positive constant $C$, depending only on $n$, such that the following assertions are true:
\begin{enumerate}
\item[\rm (i)] for any $f\in B(\rn)$ and $k_0\in\nn$,
\begin{align}\label{obj}
2^{-k_0n} \sum_{j\in J_0}\fint_{Q_{2^{-k_0}}(a_j)}|f| \le C \|f\|_{\mathrm{B}(\rn)},
\end{align}
where $\{Q_{2^{-k_0}}(a_j)\}_{j\in J_0}$ are any mutually disjoint
$2^{-k_0}$-cubes in $\rn$  with sides parallel to the coordinate axes and  $\#J_0\le 2^{k_0(n-1)}$;

\item[\rm (ii)] for any $f\in B(Q_0)$ and $k_0\in\nn$,
\begin{align}\label{obj2}
2^{-k_0n} \sum_{j\in J_0}\fint_{Q_{2^{-k_0}}(a_j)}|f| \le C \|f\|_{\mathrm{B}(Q_0)},
\end{align}
where $\{Q_{2^{-k_0}}(a_j)\}_{j\in J_0}$ are any mutually disjoint
$2^{-k_0}$-cubes in $Q_0$  with sides parallel to the coordinate axes and  $\#J_0\le 2^{k_0(n-1)}$.
\end{enumerate}

\end{proposition}

\begin{proof}
First, we show (i).
If $k_0=1$,  then $\#J_0\le 2^{n-1}$ and hence
$$2^{-k_0n} \sum_{j\in J_0}\fint_{Q_{2^{-k_0}}(a_j)}|f|\ls  \sup_{|Q|=1}\int_{Q}|f| \ls   \|f\|_{\mathrm{B}(\rn)}.$$
Below we assume that $k_0\ge 2$. Since $\#J_0\le 2^{k_0(n-1)}$, from Lemma \ref{part}(a), it follows that there exist  $2$-times extensions
of the cubes $\{Q_{2^{-k_0}}(a_j)\}_{j\in J_0}$, denoted by $\{Q_{2^{-k_0+1}}(a_{j,1})\}_{j\in J_1}$, so that
the set $\{Q_{2^{-k_0+1}}(a_{j,1})\}_{j\in J_1}$
can be divided into $N_1 \le 2^n$ subgroups, where $J_0=J_1=J_1^1\cup\cdots\cup J_1^{N_1}$.
Moreover, for each $i\in\{1,\ldots,N_1\}$,
the cubes $\{Q_{2^{-k_0+1}}(a_{j,1})\}_{j\in J_1^i}$  are mutually disjoint and
$\# J_1^i \le 2^{(k_0-1)(n-1)}$.

If $k_0-1\ge 2$,  we repeat the above procedure for the each group $\{Q_{2^{-k_0+1}}(a_{j,1})\}_{j\in J_1^i}$ with $i\in\{1,\ldots,N_1\}$,
and determine a desired collection $\{Q_{2^{-k_0+2}}(a_{j,2})\}_{j\in J_2^i}$ of $2^{-k_0+2}$-cubes, where $J_1^i=J_2^i$.
Moreover, by Lemma \ref{part}(a), we know that the set $J_2^i$ can be divided into
$N_{2,i} \le  2^n$ subgroups, denoted by $\{J_{2}^{i,1},\ldots, J_{2}^{i,N_{2,i}}\}$, such that
the cubes $\{Q_{2^{-k_0+2}}(a_{j,2})\}_{j\in J_{2}^{i,\ell}}$ for each $\ell\in\{1,\ldots,N_{2,i}\}$ are mutually disjoint
and $\# J_{2}^{i,\ell}\le  2^{(k_0-2)(n-1)}$. Write $J_2:= \cup_{i=1}^{N_1} \cup_{\ell=1}^{N_{2,i}}J_2^{i,\ell}$. Again we have $J_2=J_1=J_0$.

Iteratively,
we can find sets   $\{J_1, J_2, \ldots, J_{k_0-1}\}$ of indices, having the following properties:
for any $m\in\{1,\ldots,k_0-1\}$,

\begin{enumerate}
\item[(P-a)] $J_{k_0-1}=\cdots=J_1=J_0$;

\item[(P-b)] each $J_m$ can be written as
$$J_{m}=\bigcup_{i_1=1}^{N_1} \bigcup_{i_2=1}^{N_{2,i_1}}\cdots \bigcup_{i_{m}=1}^{N_{m,i_1,i_2,\ldots,i_{m-1}}} J_{m}^{i_1,i_2,\ldots,i_m}$$
with every $\# J_{m}^{i_1,i_2,\ldots,i_m}\le 2^{(k_0-m)(n-1)}$;

\item[(P-c)] for each $a_{j,m-1}$ with $j\in J_{m-1}^{i_1,i_2,\ldots,i_{m-1}} \subset J_{m-1}$, there exist $i_m\in \{1,\ldots, N_{m,i_1,\ldots, i_m}\}$
 and some point $a_{j',m}$ with $j'\in J_{m}^{i_1,i_2,\ldots,i_m}$ such that
 $$Q_{2^{-k_0+m-1}}(a_{j,m-1})\subset Q_{2^{-k_0+m}}(a_{j',m});$$

\item[(P-d)]  the cubes in $\{Q_{2^{-k_0+m}}(a_{j,m})\}_{j\in J_{m}^{i_1,i_2,\ldots,i_m}}$ are mutually disjoint.
\end{enumerate}
Therefore, for each point $a_j$ with $j\in J_0$, there exists a sequence  of points,
$$\lf\{a_{j_1, 1},\, a_{j_2,2},\ldots, a_{j_{k_0-1}, k_0-1}\r\},$$
such that  $j_i\in J_i$ for any $i\in\{1,\ldots, k_0-1\}$ and
$$Q_{2^{-k_0}}(a_j)\subset Q_{2^{-k_0+1}}(a_{j_1,1})\subset \cdots\subset  Q_{2^{-1}}(a_{j_{k_0-1},k_0-1}).$$
Consequently,
\begin{align*}
&\fint_{Q_{2^{-k_0}}(a_j)}|f|\\
&\quad\le
\fint_{Q_{2^{-k_0}}(a_0)}|f-f_{Q_{2^{-k_0+1}}(a_{j_1,1})}|
+\sum_{i=1}^{k_0-2}|f_{Q_{2^{-k_0+i}}(a_{j_i,i})}-f_{Q_{2^{-k_0+i+1}}(a_{j_{i+1},i+1})}|
+|f_{Q_{2^{-1}}(a_{k_0-1})} |
\\
&\quad\ls \fint_{Q_{2^{-k_0+1}}(a_{j_1,1})}|f-f_{Q_{2^{-k_0+1}}(a_{j_1,1})}|
+\sum_{i=1}^{k_0-2}\fint_{Q_{2^{-k_0+i+1}}(a_{j_{i+1},i+1})} |f-f_{Q_{2^{-k_0+i+1}}(a_{j_{i+1},i+1})}|
+\sup_{|Q|=1}\int_{Q}|f|.
\end{align*}
If $k_0=2$, then the middle term in the above summation on $i\in\{1,\ldots, k_0-2\}$ disappears.

From the above formula and $J_{k_0-1}=\cdots=J_1=J_0$,  we deduce that
\begin{align*}
2^{-k_0n}\sum_{j\in J_0}\fint_{Q_{2^{-k_0}}(a_j)}|f|
 &\ls 2^{-k_0n}\sum_{j\in J_1}\fint_{Q_{2^{-k_0+1}}(a_{j,1})}|f-f_{Q_{2^{-k_0+1}}(a_{j,1})}|\\
&\quad+2^{-k_0n}\sum_{i=1}^{k_0-2}\sum_{j\in J_{i+1}}\fint_{Q_{2^{-k_0+i+1}}(a_{j,i+1})} |f-f_{Q_{2^{-k_0+i+1}}(a_{j,i+1})}|\\
&\quad +\sharp J_0 2^{-k_0n}\sup_{|Q|=1}\int_{Q}|f|\\
&=: {\rm Z_1}+{\rm Z_2} +{\rm Z_3}.
\end{align*}
Using $J_1=\bigcup_{i=1}^{N_1} J_1^i$, $\sharp J_1^i\le 2^{(k_0-1)(n-1)}$ and  Lemma \ref{l-eqv}, we have
\begin{align*}
{\rm Z_1}
&= 2^{-k_0n}\sum_{i=1}^{N_1}\sum_{j\in J_1^i }\fint_{Q_{2^{-k_0+1}}(a_{j,1})}|f-f_{Q_{2^{-k_0+1}}(a_{j,1})}|\\
&\le 2^{-k_0n}\sum_{i=1}^{N_1} 2^{(k_0-1)(n-1)} [f]_{2^{-k_0+1}}\ls \sup_{k\in\nn}[f]_{2^{-k}}\ls \|f\|_{\mathrm{B}(\rn)}.
\end{align*}
By the above property (P-b) and Lemma \ref{l-eqv}, we obtain
\begin{align*}
{\rm Z_2}
&=2^{-k_0n}\sum_{m=2}^{k_0-1}\sum_{j\in J_{m}}\fint_{Q_{2^{-k_0+m}}(a_{j,m})} |f-f_{Q_{2^{-k_0+m}}(a_{j,m})}|\\
&=2^{-k_0n}\sum_{m=2}^{k_0-1}
\sum_{i_1=1}^{N_1} \sum_{i_2=1}^{N_{2,i_1}}\cdots \sum_{i_{m}=1}^{N_{m,i_1,i_2,\ldots,i_{m-1}}} \sum_{j\in J_{m}^{i_1,i_2,\ldots,i_m}}
\fint_{Q_{2^{-k_0+m}}(a_{j,m})} |f-f_{Q_{2^{-k_0+m}}(a_{j,m})}|\\
&\le 2^{-k_0n}\sum_{m=2}^{k_0-1}
\sum_{i_1=1}^{N_1} \sum_{i_2=1}^{N_{2,i_1}}\cdots \sum_{i_{m}=1}^{N_{m,i_1,i_2,\ldots,i_{m-1}}}
2^{(k_0-m)(n-1)} [f]_{2^{-k_0+m}}\\
&\ls  2^{-k_0n}\sum_{m=2}^{k_0-1} 2^{nm}2^{(k_0-m)(n-1)} \sup_{k\in\nn}[f]_{2^{-k}}\ls \sum_{m=2}^{k_0-1}  2^{m-k_0}\|f\|_{\mathrm{B}(\rn)}\ls \|f\|_{\mathrm{B}(\rn)}.
\end{align*}
Finally, from $\sharp J_0\le 2^{k_0(n-1)}$, it follows easily that
$$
{\rm Z_3}
\ls \sup_{|Q|=1}\int_{Q}|f| \ls \|f\|_{\rm B(\rn)}.
$$
Combining the estimates of ${\rm Z_1}$ through ${\rm Z_3}$, we obtain \eqref{obj}. This finishes the proof of (i).

Now we prove (ii).
For any $j\in J_0$, since $Q_{2^{-k_0}}(a_j)\subset Q_0$, it follows that
it intersects at most $2^n$ dyadic cubes with side length $2^{-k_0}$ in $\overline{Q}_0$.
We write these dyadic cubes as
$$\lf\{Q_{k_0,1}(a_j^1), \ldots, Q_{k_0,N_j}(a_j^{N_j})\r\},$$ where $N_j$ depends on $a_j$ and $N_j\le 2^n$. Then
\begin{align*}
 \fint_{Q_{2^{-k_0}}(a_j)}|f| \le  \sum_{i=1}^{N_j} \fint_{Q_{k_0,i}(a_j^i)}|f|.
\end{align*}
By the mutually disjointness of  $\{Q_{2^{-k_0}}(a_j)\}_{j\in J_0}$ and the geometric properties of dyadic cubes,
we know that a dyadic cube of side length $2^{-k_0}$ can intersect at most $2^n$ cubes from $\{Q_{2^{-k_0}}(a_j)\}_{j\in J_0}$,
which implies that the same dyadic cube can appear  at most $2^n$ times in the family
$$\lf\{Q_{k_0,i}(a_j^i):\, j\in J_0,\, i\in \{1,\ldots, N_j\}\,\r\}.$$
Therefore, the set $\{Q_{k_0,i}(a_j^i):\, j\in J_0,\, i\in \{1,\ldots, N_j\}\}$ can be
decomposed into $2^n$ subgroups
$$\lf\{\{Q_i\}_{i\in J_j}\r\}_{j=1}^{2^n}$$
of dyadic cubes with side length $2^{-k_0}$ in $\overline{Q}_0$, where, for any $k\in\{1,\ldots,2^n\}$,
 $\#J_k\le 2^{k_0(n-1)}$ and
 $\{Q_i\}_{i\in J_{k}}$ are mutually disjoint.
Then
\begin{align*}
 \sum_{j\in J_0} \fint_{Q_{2^{-k_0}}(a_j)}|f| \le  \sum_{k=1}^{2^n}\sum_{i\in J_k} \fint_{Q_i}|f|.
\end{align*}
For each $k\in \{1,\ldots,2^n\}$, by an argument similar to that used in the proof of (i), with  Lemma \ref{part}(a) used therein replaced by Lemma \ref{part}(b), we conclude that
\begin{align*}
2^{-k_0n}\sum_{i\in J_k} \fint_{Q_i}|f|\ls \|f\|_{\mathrm{B}(Q_0)},
\end{align*}
whence
\begin{align*}
2^{-k_0n}\sum_{j\in J_0}\fint_{Q_{2^{-k_0}}(a_j)}|f|
\ls\sum_{k=1}^{2^n} \|f\|_{\mathrm{B}(Q_0)}\ls \|f\|_{\mathrm{B}(Q_0)}.
\end{align*}
This proves \eqref{obj2}, which completes the proof of (ii) and hence of Proposition \ref{prop-new}.
\end{proof}

\section{Proof of Theorem \ref{t-bmo}}\label{s1+}
The aim of this section is to prove Theorem  \ref{t-bmo}, by first proving  (ii) $\Longleftrightarrow$ (i) $\Longleftrightarrow$ (iv), and then (iii) $\Longleftrightarrow$ (i) $\Longleftrightarrow$ (v).

We begin with the following equivalent descriptions
of Theorem \ref{t-bmo}(i), which is from \cite[Proposition 1]{bls}. Recall that
a function $g:\, \cc\to\cc$ is said to be \emph{Lipschitz continuous} if
$${\rm Lip}(g):=\sup_{x,y\in\cc,\, x\neq y}\frac{|g(x)-g(y)|}{|x-y|}<\infty.$$

\begin{lemma}\label{eqthm1}
The following are equivalent:
\begin{enumerate}
\item[{\rm(a)}] $\sup_{x,y\in \cc}(1+|x-y|)^{-1}|g(x)-g(y)|<\fz$;

\item[{\rm(b)}] there exist  positive constants $\az$ and $C$ such that
$|g(x)-g(y)|\le C$ for any $x,\,y\in\cc$ satisfying $|x-y|\le \az$;

\item[{\rm(c)}] $g$ is a sum of a bounded Borel measurable function and a Lipschitz continuous function.
\end{enumerate}
\end{lemma}

We  now use Lemma \ref{eqthm1} to prove the equivalence (ii) $\Longleftrightarrow$ (i) $\Longleftrightarrow$ (iv) of Theorem \ref{t-bmo}.

\begin{proof}[Proof of Theorem \ref{t-bmo}: (ii) $\Longleftrightarrow$ (i) $\Longleftrightarrow$ (iv)]
The implication  (ii) $\Longrightarrow$ (i) has already been proved in the introduction.
Concerning (iv) $\Longrightarrow$ (i),   we can argue in the same way as (ii) $\Longrightarrow$ (i). The only
difference lies in that \cite[Theorem 1]{bls}
deals only with $\mathop\mathrm{BMO}$ spaces on the whole Euclidean spaces. However, via checking the details,
one would find that the coresponding results in \cite{bls}
can be transferred to   $\mathop\mathrm{BMO}$ spaces on $(0,1)^n$. (One can also prove  (iv) $\Longrightarrow$
(i) in a different way:  using the trivial
fact (iv) $\Longrightarrow$ (v), and then proving  (v) $\Longrightarrow$ (i) which will be done later.)

It remains to show (i) $\Longrightarrow$ (ii) and (i) $\Longrightarrow$ (iv).
By Lemma \ref{eqthm1}, we can separately consider the case when $g$ is bounded and the case when $g$ is Lipschitz continuous.

If $g$ is bounded, then $g\circ f$ is bounded. Since $L^\infty(\rn)\hookrightarrow \mathrm{B}(\rn)$
and $L^\infty(Q_0)\hookrightarrow \mathrm{B}(Q_0)$,
it easily follows that $T_g(\mathrm{B}(\rn))\subset \mathrm{B}(\rn)$
and $T_g(\mathrm{B}(Q_0))\subset \mathrm{B}(Q_0)$.

Assume now $g$ is  Lipschitz continuous. Then, for any cube $Q$, we have
\begin{align*}
\fint_Q |g\circ f(x)-(g\circ f)_Q|\,dx
&\le \fint_Q \fint_Q|g\circ f(x)-g\circ f(y)|\,dy\,dx\le 2\,{\rm Lip}(g) \fint_Q | f(x)- f_Q|\,dx
\end{align*}
and
\begin{align*}
\fint_Q |g\circ f(x)|\,dx
&\le \fint_Q  |g\circ f(x)-g(0)|\,dx+|g(0)|\le {\rm Lip}(g) \fint_Q |f(x)|\,dx +|g(0)|,
\end{align*}
which imply that
$$\|g\circ f\|_{\mathrm{B}(\rn)}\le 2\, {\rm Lip}(g) \|f\|_{\mathrm{B}(\rn)}+|g(0)|$$
and
$$\|g\circ f\|_{\mathrm{B}(Q_0)}\le 2\, {\rm Lip}(g) \|f\|_{\mathrm{B}(Q_0)}+|g(0)|.$$
Thus, $T_g(\mathrm{B}(\rn))\subset \mathrm{B}(\rn)$
and $T_g(\mathrm{B}(Q_0))\subset \mathrm{B}(Q_0)$. We finish the proofs of  (i) $\Longrightarrow$ (ii)
and (i) $\Longrightarrow$ (iv).

Altogether, we obtain (ii) $\Longleftrightarrow$ (i) $\Longleftrightarrow$ (iv) of Theorem \ref{t-bmo}.
\end{proof}

As was mentioned in the introduction,  the proof
of the related results for ${\rm CMO}$-like spaces $\mathrm{B}_c$ could be more complicated.
However, we still benefit a lot from such proof: on the one hand, it provides \emph{an alterative way}
to show (ii) $\Longleftrightarrow$ (i) $\Longleftrightarrow$ (iv), since one has (ii) $\Longrightarrow$ (iii)
and (iv) $\Longrightarrow$ (v) trivially; on the other hand, it also provides some byproducts (see, for example, Proposition \ref{p-pm} and Lemma \ref{l-dilation} below) which are helpful for
the understanding of the spaces $\mathrm{B}$.

Recall that, for a given quasi-Banach space $X$ equipped with a
quasi-norm $\|\cdot\|_X$,  a function $h$ defined on $\rn$ is called a \emph{pointwise multiplier} on $X$
if there exists a positive constant $C$ such that $\|hf\|_{X}\le C\|f\|_X$ for any $f\in X$.
Applying Proposition \ref{prop-new}, we have the following results on the pointwise multipliers of $\rm B(\rn)$ and $\rm B(Q_0)$.
Recall that $C_c^1(\rn)$ denotes the set of all continuously differentiable functions with compact supports in $\rn$ and
$C_c^1(Q_0)$  set of all continuously differentiable functions with compact supports in $Q_0$.

\begin{proposition}\label{p-pm}
\begin{enumerate}
\item[{\rm (i)}] The elements in $C_c^1(\rn)$ are pointwise multipliers on $\mathrm{B}(\rn)$.

\item[{\rm (ii)}] The elements in $C_c^1(Q_0)$ are pointwise multipliers on $\mathrm{B}(Q_0)$.
\end{enumerate}
\end{proposition}

\begin{proof}
First, let us prove (i). Fix  $\phi\in C_c^1(\rn)$. It suffices to show that
\begin{align}\label{eq-x1}
\|\phi f\|_{\mathrm{B}(\rn)}\ls \lf[\|\phi\|_{L^\fz(\rn)}+\|\nabla\phi\|_{L^\fz(\rn)}\r] \|f\|_{\mathrm{B}(\rn)},\quad\forall\ f\in \mathrm{B}(\rn).
\end{align}
Obviously, for any   cube $Q$ with $|Q|=1$,
$$\fint_{Q}|f(x)\phi(x)|\,dx\le \|\phi\|_{L^\fz(\rn)}\fint_{Q}|f(x)|\,dx\le \|\phi\|_{L^\fz(\rn)}\|f\|_{\mathrm{B}(\rn)},\quad\forall\ f\in \mathrm{B}(\rn).$$
Next, let $k_0\in\nn$ and $\cf_{2^{-k_0}}:=\{Q_{2^{-k_0}}(a_j)\}_{j\in J_0}$
be a collection of mutually disjoint $2^{-k_0}$-cubes  in $\rn$ with $\#J_0\le 2^{k_0(n-1)}$.
Then,  for any $j\in J_0$,
\begin{align*}
&\fint_{Q_{2^{-k_0}}(a_j)}|f\phi-(f\phi)_{Q_{2^{-k_0}}(a_j)}|\\
&\quad\le \fint_{Q_{2^{-k_0}}(a_j)}|(f -f_{Q_{2^{-k_0}}(a_j)})\phi|+
\fint_{Q_{2^{-k_0}}(a_j)}|f_{Q_{2^{-k_0}}(a_j)}\phi-(f\phi)_{Q_{2^{-k_0}}(a_j)}|\\
&\quad
\le \|\phi\|_{L^\fz(\rn)}\fint_{Q_{2^{-k_0}}(a_j)}|f-f_{Q_{2^{-k_0}}(a_j)}|+\frac{\sqrt n}{2}2^{-k_0}\|\nabla \phi\|_{L^\fz(\rn)}
\fint_{Q_{2^{-k_0}}(a_j)}|f|.
\end{align*}
On the one hand, we notice that
\begin{align*}
2^{-k_0(n-1)}\sum_{j\in J_0}\fint_{Q_{2^{-k_0}}(a_j)}|f -f_{Q_{2^{-k_0}}(a_j)}|\le [f]_{2^{-k_0}}\le \|f\|_{\mathrm{B}(\rn)}.
\end{align*}
On the other hand, by Proposition \ref{prop-new}(i), we know that
\begin{align*}
2^{-k_0(n-1)} 2^{-k_0}\sum_{j\in J_0}\fint_{Q_{2^{-k_0}}(a_j)}|f| \ls \|f\|_{\mathrm{B}(\rn)}.
\end{align*}
Thus, we obtain
\begin{align*}
2^{-k_0(n-1)}\sum_{j\in J_0}\fint_{Q_{2^{-k_0}}(a_j)}|f\phi-(f\phi)_{Q_{2^{-k_0}}(a_j)}|
\ls [\|\phi\|_{L^\fz(\rn)}+\|\nabla\phi\|_{L^\fz(\rn)}] \|f\|_{\mathrm{B}(\rn)}.
\end{align*}
Further, via taking supremum over all $k_0\in\nn$ in both sides of the above inequality, we conclude that
$$\sup_{k\in\nn} [f\phi]_{2^{-k}}\ls  [\|\phi\|_{L^\fz(\rn)}+\|\nabla\phi\|_{L^\fz(\rn)}] \|f\|_{\mathrm{B}(\rn)},\quad\forall\ f\in \mathrm{B}(\rn),$$
which, together with Lemma \ref{l-eqv}(i), implies \eqref{eq-x1}. This finishes the proof of (i).

To prove (ii), we fix  $\phi\in C_c^1(Q_0)$.
It is a trivial fact that
 $$\fint_{Q_0}|f(x)\phi(x)|\,dx\le \|\phi\|_{L^\fz(Q_0)}\fint_{Q_0}|f(x)|\,dx,\quad\forall\ f\in \mathrm{B}(Q_0).$$
Similarly to the proof of (i), but now we use Proposition \ref{prop-new}(ii) to deduce that
$$\sup_{k\in\nn} [f\phi]_{2^{-k}, Q_0}\ls  [\|\phi\|_{L^\fz(\rn)}+\|\nabla\phi\|_{L^\fz(\rn)}] \|f\|_{\mathrm{B}(Q_0)}.$$
This, combined with Lemma \ref{l-eqv}(ii), implies that
$$\|\phi f\|_{\mathrm{B}(Q_0)}\ls [\|\phi\|_{L^\fz(\rn)}+\|\nabla\phi\|_{L^\fz(\rn)}] \|f\|_{\mathrm{B}(Q_0)},
\quad\forall\ f\in \mathrm{B}(Q_0),$$
which completes the proof of (ii) and hence of Proposition \ref{p-pm}.
\end{proof}

\begin{lemma}\label{l-dilation}
For any $\lz\in [1,\fz)$, there exists a positive constant $C$, depending only  $n$, such that
$$\|f(\lz\cdot)\|_{\mathrm{B}(\rn)}\le C \|f\|_{\mathrm{B}(\rn)},\quad \forall f\in \mathrm{B}(\rn).$$
\end{lemma}

\begin{proof}
For any cube $Q$, write $Q^\lz:=\{\lz x:\  x\in Q\}$, which is also a cube with the same center as that of $Q$ but of side length $\lz \ell(Q)$.  Let $L\ge0$ be the unique integer such that $2^{L-1}<\lz\le 2^{L}$.
When $|Q|=1$,  there exist $2^{Ln}$ cubes $\{Q_1, \ldots, Q_{2^{Ln}}\}$ with side length $1$ such that $Q^\lz \subset \cup_{i=1}^{2^{Ln}} Q_i$,
which implies that
\begin{align*}
\fint_Q |f(\lz x)|\,dx = \fint_{Q^\lz} |f(x)|\,dx \le \frac1{\lz^n} \sum_{i=1}^{2^{Ln}}\fint_{Q_i}|f(x)|\,dx\le 2^n \sup_{|Q|=1}\fint_Q|f(y)|\,dy,
\end{align*}
and hence
$$\sup_{|Q|=1}\fint_Q |f(\lz x)|\,dx\le 2^n \|f\|_{\mathrm{B}(\rn)}.$$
For any $\ez\in(0,1)$ and  $\cf_{\ez}:=\{Q_{\ez}(a_j)\}_{j\in J}$ being a collection
of mutually disjoint $\ez$-cubes in $\rn$ with $\#J\le \ez^{1-n}$, we have
\begin{align*}
\ez^{n-1} \sum_{j\in J} \fint_{Q_\ez(a_j)}|f(\lz x)-(f(\lz\cdot))_{Q_\ez(a_j)}|\,dx \le \ez^{n-1} \sum_{j\in J} \fint_{Q_{\ez\lz}(a_j\lz)}\fint_{Q_{\ez\lz}(a_j\lz)}|f(x)-f(y)|\,dx\,dy.
\end{align*}
When $\ez\lz\ge 1$,  similarly to the previous argument, we find that
\begin{align*}
\ez^{n-1} \sum_{j\in J} \fint_{Q_\ez(a_j)}|f(\lz x)-(f(\lz\cdot))_{Q_\ez(a_j)}|\,dx
&\le 2 \ez^{n-1} \sum_{j\in J}\fint_{Q_{\ez\lz}(a_j\lz)} |f(x)|\,dx\\
&\le  2\ez^{n-1} \sum_{j\in J} 2^n\sup_{|Q|=1}\fint_{Q}|f(x)|\,dx\le 2^{n+1} \|f\|_{\mathrm{B}(\rn)}.
\end{align*}
When $\ez\lz<1$, noticing that $\{Q_{\ez\lz}(a_j\lz)\}_{j\in J}$ are also mutually disjoint, we separate $J$
as the union of $\{J_1, \ldots, J_{2^{L(n-1)}}\}$ with each $\# J_i\le (\ez\lz)^{1-n}$ for any $i\in\{1,\ldots, 2^{L(n-1)}\}$, and
we then have
\begin{align*}
\ez^{n-1} \sum_{j\in J} \fint_{Q_\ez(a_j)}|f(\lz x)-(f(\lz\cdot))_{Q_\ez(a_j)}|\,dx
%&= \ez^{n-1} \sum_{j\in J} \fint_{Q_{\ez\lz}(a_j\lz)} |f(x)-f_{Q_{\ez\lz}(a_j\lz)}|\,dx\\
&= \ez^{n-1}  \sum_{i=1}^{2^{L(n-1)}} \sum_{j\in J_i} \fint_{Q_{\ez\lz}(a_j\lz)} |f(x)-f_{Q_{\ez\lz}(a_j\lz)}|\,dx\\
&\le \lz^{1-n}\sum_{i=1}^{2^{L(n-1)}} [f]_{\ez\lz}\le 2^{n-1} \|f\|_{\mathrm{B}(\rn)}.
\end{align*}
Summarizing all, we conclude the proof of Lemma \ref{l-dilation}.
\end{proof}

\begin{lemma}\label{lem1}
\begin{enumerate}
\item[{\rm(i)}] If $T_g[\mathrm{B}_c(\rn)]\subset \mathrm{B}(\rn)$ and $g(0)=0$,
then there exist a cube $Q\subset \rn$ and positive constants $C_1$
and $C_2$ such that $\| g\circ f\|_{\mathrm{B}(\rn)}\le C_2$
whenever $f\in \mathrm{B}_c(\rn)$ with $\supp f \subset Q$ and $\|f\|_{\mathrm{B}(\rn)}\le C_1.$

\item[{\rm(ii)}]  The conclusion in (i) is true for $\mathrm{B}(Q_0)$; that is, if $T_g[\mathrm{B}_c(Q_0)]\subset \mathrm{B}(Q_0)$ and $g(0)=0$, then there exist a cube $Q\subset Q_0$ and positive constants $C_1$
and $C_2$ such that $\| g\circ f\|_{\mathrm{B}(Q_0)}\le C_2$
whenever $f\in \mathrm{B}_c(Q_0)$ with $\supp f \subset Q$ and $\|f\|_{\mathrm{B}(Q_0)}\le C_1.$
 \end{enumerate}
\end{lemma}

\begin{proof}
By similarity, we only prove (i). We use the method of reduction to absurdity. Assume that the conclusion (i) of this lemma is false, that is,
for any cube $Q\subset\rn$ and any positive constants $C_1$
and $C_2$, there exists   $f\in \mathrm{B}_c(\rn)$ with $\supp f \subset Q$ and $\|f\|_{\mathrm{B}(\rn)}\le C_1$
such that $\| g\circ f\|_{\mathrm{B}(\rn)}>C_2$.

Let $\{Q_j\}_{j\in\nn}$ be  mutually disjoint  cubes in $\rn$.
Pick a sequence $\{\phi_j\}_{j\in\nn}\subset C_c^\fz(\rn)$ satisfying that, for any $j\in\nn$, $\phi_j\equiv 1$ on $\frac12 Q_j$
and $\phi_j\equiv 0$ out of $Q_j$. For any $j\in\nn$, by Proposition \ref{p-pm}, there exists a positive number $\gamma_j$ such that
\begin{equation}\label{x1}
\|\phi_j h\|_{\mathrm{B}(\rn)}\le \gz_j \|h\|_{\mathrm{B}(\rn)},\quad \forall\ h\in \mathrm{B}(\rn).
\end{equation}

Fix $j\in\nn$. If we take $C_1=2^{-j}$ and $C_2=j\gz_j$, then there exists $f_j\in \mathrm{B}_c(\rn)$
with $\supp f_j\subset \frac12 Q_j$ and $\|f_j\|_{\mathrm{B}(\rn)}\le 2^{-j}$ such that
$\| g\circ f_j\|_{\mathrm{B}(\rn)}>j\gz_j.$
Define $f:=\sum_{j\in\nn} f_j$, which converges in $\mathrm{B}(\rn)$.
Indeed, noticing that $f\in \mathrm{B}_c(\rn)$ and
$\|f\|_{\mathrm{B}(\rn)}\le \sum_{j\in\nn} \|f_j\|_{\mathrm{B}(\rn)}\le 1$, we then have $f(x)=\sum_{j\in\nn} f_j(x)$
for almost every $x\in\rn$, which implies that
$$f(x)=
\begin{cases}
f_j(x)\qquad &\textup{for almost every }\; x\in \frac12 Q_j,\\
0\qquad &\textup{for almost every }\; x\in Q_j\setminus(\frac12 Q_j).
\end{cases}$$
Further, from $g(0)=0$, we deduce that
$(g\circ f)\phi_j=g\circ f_j$ holds true almost everywhere.
By the assumption $T_g(\mathrm{B}_c(\rn))\subset \mathrm{B}(\rn)$,
we  know that $g\circ f\in \mathrm{B}(\rn)$. However, it follows from \eqref{x1}  that
$$j\gz_j< \| g\circ f_j\|_{\mathrm{B}(\rn)}=\|(g\circ f)\phi_j\|_{\mathrm{B}(\rn)}\le \gz_j \|g\circ f\|_{\mathrm{B}(\rn)};$$
that is, $\|g\circ f\|_{\mathrm{B}(\rn)}>j$ for any $j\in \nn$, which is a contradiction.
This finishes the proof of (i) and hence of Lemma \ref{lem1}.
\end{proof}

\begin{lemma}\label{bmof}
For any integer $j\ge 3$, there exists a non-negative function $\theta_j\in C_c^\fz(\rn)$ such that
$\theta_j(x)=1$ if $|x|\le \frac1j$, $\theta_j(x)=0$ if $|x|\ge \frac12$, $0\le \theta_j\le 1$ and  $\|\theta_j\|_{\mathop\mathrm{BMO}\,(\rn)}\le \wz C[\log_2j]^{-1}$ for some positive constant $\wz C$ independent of $j$.
\end{lemma}

\begin{proof}
The proof is similar to that of \cite[Lemma~8]{bls}. Indeed, we only need to replace the  definition of
$\theta_j$ in \cite[p.\,535]{bls} by
$$\theta_j(x):=\frac{u\lf(\log_2 (2|x|)\r)}{\log_2\frac j2},\qquad \forall\, x\in\rn,$$
where $u$ is a smooth function on $\rr$
with $0\le u\le 1$, $u\equiv 1$ on $(-\fz,-1]$ and $u\equiv0$ on $[0,\fz)$.
The remaining part of the argument is the same as that of \cite[Lemma~8]{bls}, which completes the proof of Lemma \ref{bmof}.
\end{proof}

We also need the following conclusion, which is inspired by \cite{b93} and \cite[Lemma 2]{bls}.

\begin{lemma}\label{l-he}
\begin{enumerate}

\item[{\rm(i)}] Assume that   there exist positive constants $c_1, c_2$ and $c_3\in[0,\fz)$ and a cube $K\subset\rn$ such that
\begin{align}\label{con}
\sup_{\ez\in(0,c_2)}[g\circ f]_{\ez}=\sup_{\ez\in(0,c_2)}\sup_{\cf_\ez}\lf\{\ez^{n-1}\sum_{j\in J} M(g\circ f,Q_\ez(a_j))\r\}\le c_3
\end{align}
for any function $f\in C_c^\fz(\rn)$ with $\supp f\subset K$ and $\|f\|_{\mathrm{B}(\rn)}\le c_1$,
where the supremum is taken over all $\cf_\ez:=\{Q_\ez(a_j)\}_{j\in J}$ of mutually disjoint
$\ez$-cubes in $\rn$ with cardinality $\#\cf_\ez=\#J\le 1/\ez^{n-1}$.
Then there exists a positive constant  $m$, independent of $g$ and $f$,  such that
\begin{align}\label{lip}
\sup\lf\{|g(a)-g(b)|:\ a,\,b\in\cc,\ |a-b|\le mc_1\r\}\le 4^{n+1}c_3.
\end{align}

\item[{\rm(ii)}] The corresponding conclusion of (i)  for $\mathrm{B}(Q_0)$ is also true; namely, if
$[g\circ f]_{\ez}$ (resp., $\rn$)   in (i) is replaced by $[g\circ f]_{\ez,Q_0}$ (resp., $Q_0$), then \eqref{lip} remains true.

\end{enumerate}
\end{lemma}

\begin{proof}
First, let us prove (i). Noticing that the supremum in \eqref{con} and \eqref{lip} are invariant after modulus of constants.
Without loss of generality, we may assume that $g(0)=0$; otherwise we may use $\wz g:=g-g(0)$ instead of $g$.

Observe that the norm $\|\cdot\|_{\mathrm{B}(\rn)}$ and  the term in the left-hand side of \eqref{con} are translation invariant.
By these and Lemma \ref{l-dilation}, we can assume that $K=Q_0$ via replacing $c_1$ and $c_2$
by $\az_1c_1$ and $\az_2c_2$ for some positive constants $\az_1$ and $\az_2$ depending only on $K$.
Let $a$, $b\in\cc$ satisfy
\begin{align}\label{ab}
|a-b|\le \frac{\az_1c_1}6.
\end{align}
With $\wz C$  as in Lemma \ref{bmof}, we pick an integer $j\ge 3$ satisfying
$$2^{-j}<\az_2c_2\qquad\textup{and}\qquad \frac1{\log_2j}<\frac{\az_1c_1}{2\wz C(|a|+1)}.$$
We also assume that $j$ is chosen large enough
such that the ball $B(\vec{0}_n, \frac1j)$ contains more than $2^{j(n-1)}$ disjoint $2^{-j}$-cubes, here and hereafter,  $\vec{0}_n$ denotes the origin of $\rn$
and, for any $x\in\rn$ and $r\in(0,\fz)$,  $B(x,r):=\{y\in\rn:\,|y-x|<r\}$.
Applying $\mathop\mathrm{bmo}\,(\rn) \hookrightarrow \mathrm{B}(\rn)$ and
Lemma \ref{bmof} with a translation, we know that there exists a function $\theta_j\in C_c^\fz(\rn)$ such that
$\supp\theta_j\subset Q_0$, $ \theta_j\equiv 1$ on $B((\frac12,\ldots,\frac12), \frac1j)$,  and
\begin{align}\label{est1}
|a|\,\|\theta_j\|_{\mathrm{B}(\rn)}\le |a|\,\|\theta_j\|_{\mathop\mathrm{bmo}\,(\rn)}\le   \frac{ \wz C |a|}{\log_2j} <\frac{\az_1c_1}2.
\end{align}

Based on the choice of $j$ as above, the ball $B((\frac12,\ldots,\frac12), \frac1j)$ contains more than $2^{j(n-1)}$ disjoint $2^{-j}$-cubes.
We select $2^{j(n-1)}$ such cubes, denoted by $\{Q_i\}_{i=1}^{2^{j(n-1)}}$.
Then $\theta_j\equiv1$ on all such $Q_i$ with $i\in\{1,\ldots,2^{j(n-1)}\}$.
Denote by $k_{i}$ the lower-left-corner point of $\frac14Q_i$.
Then $$\frac14Q_i=k_i+(0,2^{-j-2})^n.$$
Choose $\phi\in C_c^\fz(\rn)$ satisfying
$\supp\phi\subset Q_0$, $\phi\equiv1$ on $(1/4,1/2)^n$ and $0\le\phi\le 1$.
Define
$$f(x):=(b-a)\sum_{i=1}^{2^{j(n-1)}}\phi\lf(2^{j+1}(x-k_i)\r)+a\theta_j(x),\quad\forall\; x\in\rn.$$
Then $f\in C_c^\fz(\rn)$ and $\supp f \subset Q_0$. Moreover, for any $i\in\{1,\ldots, 2^{j(n-1)}\}$,
$f\equiv b$ on $(\frac14Q_i)\setminus (\frac18 Q_i)$ and $f\equiv a$ on $Q_i\setminus (\frac12Q_i)$,
which consequently implies that $g\circ f\equiv g(b)$ on $(\frac14Q_i)\setminus (\frac18 Q_i)$ and $g\circ f\equiv g(a)$ on $Q_i\setminus ( \frac12Q_i)$.
Noticing that $\supp \phi(2^{j+1}(\cdot-k_i)) \subset Q_i$ and $\{Q_i\}_{i=1}^{2^{j(n-1)}}$ are mutually disjoint,
we apply \eqref{est1} and \eqref{ab}
to deduce that
\begin{align*}
\|f\|_{\mathrm{B}(\rn)}
&\le |b-a| \lf\|\sum_{i=1}^{2^{j(n-1)}}\phi(2^{j+1}(\cdot-k_i))\r\|_{\mathrm{B}(\rn)}
+|a|\|\theta_j\|_{\mathrm{B}(\rn)}\\
&\le 3|b-a| \lf\|\sum_{i=1}^{2^{j(n-1)}}\phi(2^{j+1}(\cdot-k_i))\r\|_{L^\fz(\rn)}
+|a|\|\theta_j\|_{\mathrm{B}(\rn)}< 3|b-a|+\frac{\az_1c_1}2<\az_1c_1.
\end{align*}
Further, by the above discussion, \eqref{con} and the fact $g(0)=0$,
we conclude that
\begin{align*}
&|g(b)-g(a)|\\
&\quad\le   2^{-j(n-1)}\sum_{i=1}^{2^{j(n-1)}}\lf|\fint_{(\frac14Q_i)\setminus (\frac18 Q_i)}g \circ f(x)\,dx-\fint_{Q_i\setminus (\frac12Q_i)} g\circ f(x)\,dx\r|\\
&\quad\le 2^{-j(n-1)}\sum_{i=1}^{2^{j(n-1)}}\lf[\fint_{(\frac14Q_i)\setminus (\frac18 Q_i)}|g \circ f(x)-(g \circ f)_{Q_i}|\,dx
+\fint_{Q_i\setminus (\frac12Q_i)} |g\circ f(x)-(g \circ f)_{Q_i}|\,dx\r]\\
&\quad\le 2^{-j(n-1)}\sum_{i=1}^{2^{j(n-1)}} \lf[\frac1{4^{-n}-8^{-n}}+\frac1{1-2^{-n}}\r]
\fint_{Q_i}|g \circ f(x)-(g \circ f)_{Q_i}|\,dx\\
&\quad\le 4^{n+1} 2^{-j(n-1)}\sum_{i=1}^{2^{j(n-1)}} M(g\circ f, Q_i)\le 4^{n+1}c_3.
\end{align*}
This proves the desired conclusion of (i) with $m=1/6$.

The proof of (ii) is similar to that of (i), so we  omit the details here.
This finishes the proof of Lemma \ref{l-he}.
\end{proof}

We are now ready to prove the remainder case of Theorem \ref{t-bmo}.

\begin{proof}[Proof of Theorem \ref{t-bmo}: (iii) $\Longleftrightarrow$ (i) $\Longleftrightarrow$ (v)]
The implications (i) $\Longrightarrow$ (iii) and (i) $\Longrightarrow$ (v) follow from the previous proved equivalence of
Theorem \ref{t-bmo} and the trivial facts (ii) $\Longrightarrow$ (iii) and (iv) $\Longrightarrow$ (v).

Conversely, assume that (iii) or (v) holds true. Via subtracting $g(0)$ if necessary, we may assume that $g(0)=0$.
Then, by Lemmas \ref{lem1} and \ref{l-he}, we conclude that $g$ satisfies Lemma \ref{eqthm1}(b), and
hence (i) holds true. This proves (iii) $\Longrightarrow$ (i) and (v) $\Longrightarrow$ (i),
and then finishes the proof of Theorem \ref{t-bmo}.
\end{proof}

\begin{remark}
In  Theorem \ref{t-bmo}, by the trivial facts (ii) $\Longrightarrow$ (iii) and (iv) $\Longrightarrow$ (v), the proofs of (iii) $\Longleftrightarrow$ (i) $\Longleftrightarrow$ (v) provide an alterative way to prove  (ii) $\Longrightarrow$ (i) and (v) $\Longrightarrow$ (i).
\end{remark}

As an immediate consequence of Theorem \ref{t-bmo}, we have the following result.

\begin{corollary}
The following statements are equivalent:
\begin{enumerate}
\item[{\rm(i)}] $\sup_{x,y\in \cc}(1+|x-y|)^{-1}|g(x)-g(y)|<\fz$;
\item[{\rm(ii)}] $T_g(\mathrm{B}_0(\rn)) \subset \mathrm{B}(\rn)$;
\item[{\rm(iii)}]  $T_g(\mathrm{B}_0(Q_0)) \subset \mathrm{B}(Q_0)$.
\end{enumerate}
\end{corollary}

\section{Proof of Theorem \ref{t-vmo}}

   We give the proof of Theorem \ref{t-vmo}
in this section. To this end,
we need the following well-known fact on the relation between uniformly continuous functions and modulus of continuity (see \cite[Chapter 2, Section 6]{dl}). Recall that a function $w: [0,\fz)\to[0,\fz)$  is called a \emph{modulus of continuity} of a function $g$ provided that
$$ \lim_{t\to0}w(t)=0 \quad \mbox{and}\quad |g(x)-g(y)|\le w(|x-y|),\quad\forall\ x,\,y\in\cc.$$

\begin{lemma}\label{uni}
If a function $g$  is uniformly continuous, then it has concave increasing  modulus of continuity.
\end{lemma}

Observe that $\mathrm{B}_0(\rr)\subset \mathrm{B}_0(\rn)$ in the sense of \eqref{extend}.
Similarly to Theorem \ref{t-bmo}, we can prove
(ii) $\Longleftrightarrow$ (i) $\Longleftrightarrow$ (iv) of Theorem \ref{t-vmo}
in a rather simple way via reducing to the one-dimensional case.

\begin{proof}[Proof of  Theorem \ref{t-vmo}: (ii) $\Longleftrightarrow$ (i) $\Longleftrightarrow$ (iv)]
First we show (i) $\Longrightarrow$ (ii).
Let $g$ be a uniformly continuous function on $\cc$, and $w$   its related concave increasing
modulus of continuity, whose existence is due to Lemma \ref{uni}.
For any $f\in \mathrm{B}(\rn)$, we have
$$\sup_{|Q|=1}\fint_Q |g\circ f|\le \sup_{|Q|=1}\fint_Q w(|f(x)|)\,dx+|g(0)|\le w\lf(\sup_{|Q|=1}\fint_Q |f|\r)+|g(0)|<\fz.$$
For any $f\in \mathrm{B}(\rn)$ and any collection $\mathcal{F}_\ez$ of disjoint $\ez$-cubes $Q$ in $\rn$
with $\#\cf_\ez\le \ez^{1-n}$, by the Jensen inequality, we find that
\begin{align*}
\ez^{n-1}\sum_{Q\in\cf_\ez}\fint_Q|g\circ f-(g\circ f)_Q|
&\le\ez^{n-1}\sum_{Q\in\cf_\ez} \fint_Q \fint_Q|g\circ f(x)-g\circ f(y)|\,dx\,dy\\
&\le\ez^{n-1}\sum_{Q\in\cf_\ez} \fint_Q \fint_Q w(|f(x)-f(y)|)\,dx\,dy\\
&\le w\lf(\ez^{n-1}\sum_{Q\in\cf_\ez} \fint_Q \fint_Q| f(x)- f(y)|\,dx\,dy\r)\\
&\le w\lf(2\ez^{n-1}\sum_{Q\in\cf_\ez} \fint_Q | f- f_Q|\r).
\end{align*}
From this, it follows that
$$\lim_{\delta\to0}\sup_{\ez\in(0,\delta)}[g\circ f]_{\ez}\le \lim_{\delta\to0}w\lf(2\sup_{\ez\in(0,\delta)} [f]_\ez\r)=0,\quad \forall\;f\in \mathrm{B}_0(\rn).$$
This proves that $T_g(\mathrm{B}_0(\rn))\subset\mathrm{B}_0(\rn)$ and hence (i) $\Longrightarrow$ (ii).
The proof of (i) $\Longrightarrow$ (iv) is similar, and we omit its details.

Taking into account the facts that
$\mathop\mathrm{vmo}\,(\rr)=\mathrm{B}_0(\rr)\subset \mathrm{B}_0(\rn)$ and
$${\rm VMO}\,((0,1))=\mathrm{B}_0((0,1))\subset \mathrm{B}_0((0,1)^n)$$
in the sense of \eqref{extend} as well as
\cite[Theorem 2]{bls},
the implications (ii) $\Longrightarrow$ (i) and  (iv) $\Longrightarrow$ (i)
can be proved in the same way as those of (ii) $\Longrightarrow$ (i) and (iv) $\Longrightarrow$ (i) in
Theorem \ref{t-bmo}.

These prove  (ii) $\Longleftrightarrow$ (i) $\Longleftrightarrow$ (iv) of  Theorem \ref{t-vmo}.
\end{proof}

To show the other implications of
Theorem \ref{t-vmo}, we need the following two lemmas, which are inspired by \cite[Lemma 3]{bls}.

\begin{lemma}\label{l-unif1}
 If $T_g(\mathrm{B_c}(\rn))\subset \mathrm{B}_0(\rn)$ and $g(0)=0$, then, for any $\ez\in(0,1)$, there exist  a cube $P\subset Q_0$
and two positive constants $c_1$ and $c_2$ such that
$$\sup_{\delta\in(0,c_2]}\sup_{\mathcal{F}_\delta} \delta^{n-1} \sum_{Q\in \mathcal{F}_\delta}\fint_Q |g\circ f-(g\circ f)_Q|\le \ez$$
for any $f\in C_c^\fz(\rn)$ with $\supp f\subset P$ and $\|f\|_{\mathrm{B}(\rn)}\le c_1,$
where the second supremum is taken over all collections $\mathcal{F}_\delta$ of disjoint $\delta$-cubes with $\#\mathcal{F}_\delta \le \delta^{1-n}$.
\end{lemma}

\begin{proof}
We use the method of reduction to absurdity. Assume that there exists $\ez_0\in(0,1)$ such that, for any cube $P\subset Q_0$
and any pair $(c_1,c_2)$ of positive numbers, there exist a function $f\in C_c^\fz(\rn)$
supported in the cube $P$ and satisfying
$\|f\|_{\mathrm{B}(\rn)}\le c_1$, and a collection $\mathcal{F}_\delta$ of disjoint $\delta$-cubes with
$\delta\le c_2$ and $\#\mathcal{F}_\delta \le \delta^{1-n}$ such that
$$ \delta^{n-1} \sum_{Q\in \mathcal{F}_\delta}\fint_Q |g\circ f-(g\circ f)_Q|\ge \ez_0.$$
In particular, for any integer $j\ge 9$, associated to the cube
$$P_j:=\lf(0,\; 2^{-1}(1+j)^{-2}\r)^n +\frac{1}{j}(1,\ldots,1)\subset (2^{-j}, 1-2^{-j})^n\subset Q_0,$$
there exist   $f_j\in C_c^\fz(\rn)$ supported in the cube $\frac12P_j$ and satisfying $\|f_j\|_{\mathrm{B}(\rn)}\le 2^{-j}$,
as well as a collection $\mathcal{F}_{\delta_j}:=\{Q_{j,i}\}_i$ of disjoint $\delta_j$-cubes
with $\delta_j\le 2^{-j}$ and $\#\mathcal{F}_{\delta_j}\le \delta_j^{1-n}$, such that
$$ \delta_j^{n-1} \sum_{Q_{j,i}\in \mathcal{F}_{\delta_j}}\fint_{Q_{j,i}} |g\circ f_j-(g\circ f_j)_{Q_{j,i}}|\ge \ez_0.$$

Notice that $P_j\cap P_i=\emptyset$ whenever $i\neq j$ and $i,j\ge 9$.
Pick $\phi\in C_c^\fz(\rn)$
with $\supp\phi\subset Q_0$, $0\le \phi\le 1$ and $\phi\equiv 1$ on $\frac12Q_0$.
Define $\phi_j(x):= \phi(2(j+1)^2(x-c_{P_j}))$ for any $j\ge 9$ and $x\in\rn$, where $c_{P_j}:=\frac{1}{j}(1,\ldots,1)$ is the center of the cube $P_j$.
Then $\supp\phi_j \subset P_j$, $\supp\phi_j\equiv1$ on $\frac12P_j$ and
$$\|\nabla \phi_j\|_{L^\fz(\rn)}=2(j+1)^2\|\nabla\phi\|_{L^\fz(\rn)}.$$

Since $g(0)=0$ and $\supp f_j\subset \frac12P_j$, we may assume that $Q_{j,i}\cap P_j\neq \emptyset$ for any $Q_{j,i}\in \mathcal{F}_{\delta_j}$.
Such an assumption implies that those $Q_{j,i}$ are close to $P_j$. Meanwhile, notice that the side length of each $Q_{j,i}$ is far less than that of $P_j$.
Consequently, we find that each $Q_{j,i}\subset Q_0$ and that $Q_{j,i}\cap Q_{\ell,k}=\emptyset$ for any $i$
and $k$ whenever $j\neq \ell$ and $j,\ell\ge 9$.

Define $f:=\sum_{j=9}^\fz f_j$. Then $f\in C_c^\fz(\rn)\subset \mathrm{B}_c(\rn)$, and hence $g\circ f\in \mathrm{B}_0(\rn)$.
For any $j\ge 9$, by $g(0)=0$, $\supp f_j\subset \frac12 P_j$,  $\phi_j\equiv 1$ on $\frac 12P_j$  and $f(x)= f_j(x)$ for almost every $x\in P_j$,  we have
$$(g\circ f)\phi_j=g\circ f_j\qquad \textup{for almost every} \, x\in \rn.$$
Thus,
\begin{align}\label{est2}
\ez_0&\le\delta_j^{n-1} \sum_{Q_{j,i}\in \mathcal{F}_{\delta_j}}\fint_{Q_{j,i}} |g\circ f_j-(g\circ f_j)_{Q_{j,i}}|\\
&=\delta_j^{n-1} \sum_{Q_{j,i}\in \mathcal{F}_{\delta_j}}\fint_{Q_{j,i}} |(g\circ f)\phi_j-((g\circ f)\phi_j)_{Q_{j,i}}|\noz\\
&\le \delta_j^{n-1} \sum_{Q_{j,i}\in \mathcal{F}_{\delta_j}}\left[\|\phi_j\|_{L^\fz(\rn)}\fint_{Q_{j,i}}|g\circ f -(g\circ f)_{Q_{j,i}}|+\frac{\sqrt n}{2}\delta_j\|\nabla \phi_j\|_{L^\fz(\rn)}
\fint_{Q_{j,i}}|g\circ f|\r]\noz\\
&\le \delta_j^{n-1} \sum_{Q_{j,i}\in \mathcal{F}_{\delta_j}}\left[\|\phi_j\|_{L^\fz(\rn)}\fint_{Q_{j,i}}|g\circ f -(g\circ f)_{Q_{j,i}}|+\sqrt n\delta_j(j+1)^2\|\nabla \phi\|_{L^\fz(\rn)}
\fint_{Q_{j,i}}|g\circ f|\r]\noz\\
&\le \delta_j^{n-1} \sum_{Q_{j,i}\in \mathcal{F}_{\delta_j}} \fint_{Q_{j,i}}|g\circ f -(g\circ f)_{Q_{j,i}}|+\sqrt n \|\nabla \phi\|_{L^\fz(\rn)}\|g\circ f\|_{L^\fz(Q_0)} 2^{-j}(j+1)^2.\noz
\end{align}
Notice that $T_g(\mathrm{B}_c(\rn))\subset \mathrm{B}_0(\rn)$ implies that
$T_g(\mathrm{B}_c(\rn))\subset \mathrm{B}(\rn)$. Thus, by Theorem \ref{t-bmo} and Lemma \ref{eqthm1},
 $g$ can be written as the sum of a bounded Borel measurable function and a Lipschitz continuous function,
 both map a bounded set in $\cc$ into a bounded set. From this observation and the fact that $f\in C_c^\fz(Q_0)$,
 we deduce that $\|g\circ f\|_{L^\fz(Q_0)}$ is finite.
Then, by taking $j$ large enough in \eqref{est2}, we conclude that
\begin{align*}
\frac{\ez_0}2\le \delta_j^{n-1} \sum_{Q_{j,i}\in \mathcal{F}_{\delta_j}} \fint_{Q_{j,i}}|g\circ f -(g\circ f)_{Q_{j,i}}|,
\end{align*}
which contradicts to the fact  $g\circ f\in \mathrm{B}_0(\rn)$.
This finishes the proof of Lemma \ref{l-unif1}.
\end{proof}

An argument similar to that used in the proof of  Lemma \ref{l-unif1} gives its following counterpart,
which is also used in the proof of Theorem \ref{t-vmo}; we omit the details.

\begin{lemma}\label{l-unif2} If $T_g(\mathrm{B_c}(Q_0))\subset \mathrm{B}_0(Q_0)$ and $g(0)=0$, then, for any $\ez\in(0,1)$, there exist a cube $P\subset Q_0$
and two positive constants $c_1$ and $c_2$ such that
$$\sup_{\delta\in(0,c_2]}\sup_{\mathcal{F}_\delta} \delta^{n-1} \sum_{Q\in \mathcal{F}_\delta}\fint_Q |g\circ f-(g\circ f)_Q|\le \ez$$
for any $f\in C_c^\fz(Q_0)$ with $\supp f\subset P$, $\|f\|_{\mathrm{B}(Q_0)}\le c_1,$
where the second supremum is taken over all collections $\mathcal{F}_\delta$ of disjoint $\delta$-cubes in $Q_0$ with $\#\mathcal{F}_\delta \le \delta^{1-n}$.
\end{lemma}

We now complete the proof of Theorem \ref{t-vmo}.

\begin{proof}[Proof of Theorem \ref{t-vmo}: (iii) $\Longleftrightarrow$ (i) $\Longleftrightarrow$ (v)]
Observe that the implications (ii) $\Longrightarrow$ (iii) and (iv) $\Longrightarrow$ (v) are trivial. Therefore,
we have (i) $\Longrightarrow$ (iii) and (i) $\Longrightarrow$ (v) due to the previous  proved implications
(ii) $\Longleftrightarrow$ (i) $\Longleftrightarrow$ (iv).

Now we show  (iii) $\Longrightarrow$ (i) and   (v) $\Longrightarrow$ (i).
Without loss of generality, we may assume $g(0)=0$, by possibly subtracting $g(0)$.
If $T_g(\mathrm{B}_c(\rn))\subset\mathrm{B}_0(\rn)$ ({resp.}, $T_g(\mathrm{B}_c(Q_0))\subset\mathrm{B}_0(Q_0)$),
then the uniformly  continuity of $g$ in (i) follows from
Lemmas \ref{l-he} and \ref{l-unif1} ({resp.},  \ref{l-unif2}).
This finishes the proof of Theorem \ref{t-vmo}.
\end{proof}

\begin{remark}
In Theorem \ref{t-vmo}, by the trivial facts (ii) $\Longrightarrow$ (iii) and (iv) $\Longrightarrow$ (v), the proofs of the implications (iii) $\Longleftrightarrow$ (i) $\Longleftrightarrow$ (v) provide an alterative way to prove  (ii) $\Longrightarrow$ (i) and (v) $\Longrightarrow$ (i).
\end{remark}

\section{Proof of Theorem \ref{t-cmo}}

  To show Theorem \ref{t-cmo},
we need Theorem \ref{t-vmo} and the following result on the continuity of $T_g$.

\begin{proposition}\label{p-bmo-c}
If $g$ is uniformly continuous, then
$T_g$ is  continuous at $f\in \mathrm{B}_0(\rn)$  (resp., $\mathrm{B}_0(Q_0)$) as
a map from $\mathrm{B}(\rn)$ (resp., $\mathrm{B}(Q_0)$) to itself. \end{proposition}

The proof of Proposition \ref{p-bmo-c} replies on the following conclusion from \cite[Lemma 4]{bls}.

\begin{lemma}\label{l-bls}
Assume that $g$ has a concave increasing modulus of continuity $w$ satisfying $w(t)\to 0$ as $t\to0$. Then,
for any locally integrable functions $f$ and $h$, and any cube $Q$,
\begin{align*}
&\fint_Q\left|g\circ(f+h)-g\circ f -(g\circ(f+h)-g\circ f)_Q\right|\\
&\quad\le \min\left\{2w\left(2\fint_Q|f-f_Q|\r)+w\left(2\fint_Q|h-h_Q|\r),\, 2w\left(\fint_Q|h|\r)\right\}.
\end{align*}
\end{lemma}

By Lemma \ref{l-bls}, the proof of Proposition \ref{p-bmo-c} is similar to that of \cite[Proposition 2]{bls}, and we give some details here for completeness.

\begin{proof}[Proof of Proposition \ref{p-bmo-c}]
Due to similarity, we only consider the case when  $f\in \mathrm{B}_0(\rn)$.
Let $\delta\in(0,1)$ and $w$ be a related concave increasing modulus of continuity of $g$. Define
$$M_\delta:=\sup_{\ez\in(0,\delta)}[f]_\ez.$$ Then, for any $\varepsilon>0$,
there exists a $\delta\in(0,1/2)$ such that $w(2M_{\delta})<\varepsilon$, due to $f\in \mathrm{B}_0(\rn)$
and $\lim_{t\to0}w(t)=0$. Using $\lim_{t\to0}w(t)=0$ again, we can take $\eta>0$ such that $w(\eta/\delta^n)<\varepsilon$.

Assume now $h\in \mathrm{B}(\rn)$ satisfying that  $\|h\|_{ \mathrm{B}(\rn)}\le\eta$.
Then,  for any collection $\cf_\ez$ of disjoint $\ez$-cubes,  by Lemma \ref{l-bls} and
the Jensen inequality, we find that, when $\ez\in(0, \delta]$,
\begin{align*}
I_\ez:=&\,\ez^{n-1}\sum_{Q\in \cf_\ez}\fint_Q\left|g\circ(f+h)-g\circ f -(g\circ(f+h)-g\circ f)_Q\right|\\
\le&\,  2w\left(2\ez^{n-1}\sum_{Q\in \cf_\ez}\fint_Q|f-f_Q|\r)+w\left(2\ez^{n-1}\sum_{Q\in \cf_\ez}\fint_Q|h-h_Q|\r)\\
\le&\, 2w(2M_\delta)+w(2\|h\|_{\mathrm{B}(\rn)})<2\varepsilon+w(2\eta)<2\varepsilon+w(\eta/\delta^n)<3\varepsilon,
\end{align*}
while, when  $\ez\in(\delta,1)$,
\begin{align*}
I_\ez& \le 2\ez^{n-1}\sum_{Q\in \cf_\ez} w\left(\fint_Q|h|\r)
  \le  2\ez^{n-1}\sum_{Q\in \cf_\ez} w\left(\delta^{-n}\|h\|_{\mathrm{B}(\rn)}\r)<2\varepsilon.
\end{align*}
Furthermore, for any cube $Q$ with $|Q|=1$, by Lemma \ref{uni} and the Jensen inequality, we have
$$\fint_Q\left|g\circ(f+h)-g\circ f \right|
\le \fint_Q w(|h|)\le w\left(\fint_Q |h |\r)\le w(\|h\|_{\mathrm{B}(\rn)})\le w(\eta)<\varepsilon.$$
Altogether, we conclude that $\|T_g(f+h)-T_gf\|_{\mathrm{B}(\rn)}\to 0$ as $\|h\|_{\mathrm{B}(\rn)}\to 0$, as desired.
This finishes the proof of Proposition \ref{p-bmo-c}.
\end{proof}

Now we use Proposition \ref{p-bmo-c} to show Theorem \ref{t-cmo}.

\begin{proof}[Proof of Theorem \ref{t-cmo}]
Let us first prove (a).  On the one hand, if $T_g(\mathrm{B}_c(\rn))\subset \mathrm{B}_c(\rn)$, then $T_g(\mathrm{B}_c(\rn))\subset \mathrm{B}_0(\rn)$
and  hence $g$ is uniformly continuous in terms of Theorem \ref{t-vmo}.
On the other hand,
$T_g(\mathrm{B}_c(\rn))\subset \mathrm{B}_c(\rn)$ also implies that $g(0)=T_g(0)\in \mathrm{B}_c(\rn)$.
Notice that a constant function in $\mathrm{B}_c(\rn)$ must be zero. Thus, we have $g(0)=0$.

Conversely, we assume that $g$ is uniformly continuous and $g(0)=0$. By Theorem \ref{t-vmo} and Proposition \ref{p-bmo-c}, we know that $T_g$ is continuous from $\mathrm{B}_c(\rn)$ to $\mathrm{B}_0(\rn)$.
Moreover, when $f\in C_c^\fz(\rn)$, the condition $g(0)=0$ ensures that $g\circ f$ is a continuous function with compact support, and hence it is
a uniform limit of a sequence of functions in $C_c^\fz(\rn)$. This implies that $g\circ f\in \mathrm{B}_c(\rn)$
whenever $f\in C_c^\fz(\rn)$. From these observations, we deduce that $T_g(\mathrm{B}_c(\rn))\subset \mathrm{B}_c(\rn)$.

The proof of (b) is almost the same as that of (a); the only difference is that we need to show that any constant function $C$
belongs to the space $\mathrm{B}_c(Q_0)$. It suffices to prove that, for any $\varepsilon\in(0,1)$,
there exists a $\phi\in C_c^\fz(Q_0)$ such that
$\|C-\phi\|_{L^1(Q_0)}<\varepsilon.$ To see this, without loss of generality, we may assume that $C>0$. Pick $\delta>0$
such that $1-(1-2\delta)^n<\frac{\varepsilon}{2C}.$ Then we choose a smooth function $\phi$ such that $\supp\phi\subset (1-\delta)Q_0$,
$0\le \phi\le C$, $\phi\equiv C$ on $(1-2\delta)Q_0$, and it is easy to see
that
$$\|C-\phi\|_{L^1(Q_0)}=\|C-\phi\|_{L^1(Q_0\setminus (1-\delta)Q_0)}\le 2C [1-(1-2\delta)^n]<\varepsilon.$$
This finishes the proof of Theorem \ref{t-cmo}.
\end{proof}

\section{Proof of Theorem \ref{t-con2}}

\begin{proof}[Proof of Theorem \ref{t-con2}]
The only non-trivial part is to show ``the continuity of
$T_g$ on $\mathrm{B}(\rn)$ $\Longrightarrow$ $g$ is $\rr$-affine". However, it is also not complicated
since we can argue as in the proof of (ii) $\Longrightarrow$ (i) of Theorem \ref{t-bmo}, namely, to reduce it
to the one-dimensional case.

Indeed, assume that $T_g$ is continuous on $\mathrm{B}(\rn)$, that is, for any $f\in \mathrm{B}(\rn)$
and $\ez>0$, there exists a $\eta>0$ such that
$$\lf\|g\circ(f+h)-g\circ f\r\|_{\mathrm{B}(\rn)}<\ez$$
for any $h\in \mathrm{B}(\rn)$ with $\|h\|_{\mathrm{B}(\rn)}<\eta$.

Now, let  $f_0,\ h_0\in \mathrm{B}(\rr)=\mathop\mathrm{bmo}\,(\rr)$ and $\|h_0\|_{\mathrm{B}(\rr)}<\eta$.
Define $f$ and $h$, respectively, via $f_0$ and $h_0$ in the same way as \eqref{extend}.
Then $f,\,h\in \mathrm{B}(\rn)$ with $\|h\|_{\mathrm{B}(\rn)}\le \|h_0\|_{\mathrm{B}(\rr)}<\eta$,
and
\begin{align*}
\|g\circ(f_0+h_0)-g\circ f_0\|_{\mathop\mathrm{bmo}\,(\rr)}
=\|g\circ(f_0+h_0)-g\circ f_0\|_{{\mathrm{B}}(\rr)}=\|g\circ(f+ h)-g\circ f\|_{{\mathrm{B}}(\rn)}<\ez.
\end{align*}
This means that $T_g$ is continuous on $\mathop\mathrm{bmo}\,(\rr)$. Applying \cite[Theorem 4]{bls}, we know that
$g$ is $\rr$-affine. This proves the desired conclusion of Theorem \ref{t-con2}.
\end{proof}

\bigskip

\noindent {Liguang Liu}

\smallskip

\noindent School of Mathematics, Renmin University
of China, Beijing 100872,  People's Republic of China

\noindent {\it E-mail}: \texttt{liuliguang@ruc.edu.cn}

\bigskip

\noindent Dachun Yang and Wen Yuan (Corresponding author)

\smallskip

\noindent Laboratory of Mathematics and Complex Systems (Ministry of
Education),  School of Mathematical Sciences, Beijing Normal University,
Beijing 100875, People's Republic of China

\smallskip

\noindent{\it E-mails}: \texttt{dcyang@bnu.edu.cn} (D. Yang);

\noindent\phantom{{\it E-mails:}} \texttt{wenyuan@bnu.edu.cn} (W. Yuan)

\end{document}